\documentclass{article}
\usepackage[english]{babel}
\usepackage[letterpaper,top=1in,bottom=1in,left=1.15in,right=1.15in]{geometry}

\usepackage{amsmath,amssymb,amsfonts}
\usepackage{graphicx}
\usepackage[colorlinks=true, allcolors=blue]{hyperref}
\usepackage{enumerate}
\usepackage[font=small]{caption}
\usepackage{subcaption}
\usepackage{wrapfig}
\usepackage{xcolor}

\newenvironment{proof}{\paragraph*{\small Proof.}}{\hfill$\square$\vspace{1em}}

\usepackage[style=ieee, sorting=nty]{biblatex}
\def\cA{{\mathcal A}}

\def\cC{{\mathcal C}}

\def\cL{{\mathcal L}}
\def\cM{{\mathcal M}}

\def\cQ{{\mathcal Q}}

\def\cX{{\mathcal X}}

\def\E{\mathbb{E}}

\def\N{\mathbb{N}}

\def\P{\mathbb{P}}
\def\R{\mathbb{R}}

\def\T{\mathbb{T}}
\def\Z{\mathbb{Z}}

\renewcommand{\d}{\mathrm{d}}
\newtheorem{Def}{Definition}[section]
\newtheorem{Thm}[Def]{Theorem}
\newtheorem{Cor}[Def]{Corollary}
\newtheorem{Rmk}[Def]{Remark}
\newtheorem{Lem}[Def]{Lemma}
\newtheorem{Prop}[Def]{Proposition}

\def\ig{\mathfrak{L}}

\providecommand{\keywords}[1]
{
  \small\noindent
  \textbf{Keywords:} #1
}

\providecommand{\classification}[1]
{
  \small\noindent	
  \textbf{Mathematics Subject Classification:} #1
}

\usepackage{fancyhdr}
\begin{document}
\title{Synchronization Games\thanks{
The authors would like to thank Professor
Marco Cirant for insightful comments.
This work was partially supported 
by the National Science Foundation grant
DMS-2406762. }}
\author{Felix H\"{o}fer and H.~Mete Soner
\thanks{Both authors are with Princeton University, Department of Operations Research and Financial Engineering (e-mails: \texttt{fhoefer@princeton.edu}, \texttt{soner@princeton.edu}).}}

\maketitle

\begin{abstract}
\noindent
We propose a new mean-field game model with two states to study synchronization phenomena, and we provide a comprehensive characterization of stationary and dynamic equilibria along with their stability properties. The game undergoes a phase transition with increasing interaction strength. In the subcritical regime, the uniform distribution, representing incoherence, is the unique and stable stationary equilibrium. Above the critical interaction threshold, the uniform equilibrium becomes unstable and there is a multiplicity of stationary equilibria that are self-organizing. Under a discounted cost, dynamic equilibria spiral around the uniform distribution before converging to the self-organizing equilibria. With an ergodic cost, however, unexpected periodic equilibria around the uniform distribution emerge.
\end{abstract}

\vspace{1em}

\classification{34C25, 34H05, 37G35, 49L20, 91A16, 92B25}

\section{Introduction} \label{sec:introduction}

Building on Winfree's work in the 1960s, the Kuramoto model \cite{kuramoto_self-entrainment_1975} has become the corner stone of mathematical models of collective synchronization and has received attention in all natural sciences, engineering, and mathematics. The model consists of coupled oscillators that exhibit spontaneous synchronization once the coupling strength exceeds a critical threshold. While the classical model postulates the dynamics of each oscillator in the form of a system of nonlinear ordinary differential equations, we use the mean-field game (MFG) formalism. Indeed, instead of positing the dynamics of the particles, we let the individual particles determine their behavior endogenously by minimizing a cost functional and settling in a Nash equilibrium. More generally, MFGs were independently introduced  by Lasry \& Lions \cite{lasry_jeux_2006-1, lasry_jeux_2006,lasry_mean_2007} and Huang, Caines, \& Malhamé \cite{huang_invariance_2007, huang_large-population_2007, huang_nash_2007,huang_individual_2003} to approximate large population games in which the interaction appears through the empirical distribution of all agents. 

The MFG approach to synchronization was first proposed by Yin, Mehta, Meyn, \& Shanbhag in \cite{yin_synchronization_2010,YMMS}, and later used by Carmona \& Graves \cite{carmona_jet_2020} to study jet-lag recovery by modeling the alignment with the circadian rhythm. These studies and the recent work of Carmona, Cormier, \& Soner \cite{carmona_synchronization_2023} establish phase transitions in the Kuramoto MFGs analogous to the one exhibited by the original Kuramoto dynamical system, providing strong evidence for a connection between these two seemingly very different models.

In the mean-field game, each oscillator is treated as a rational agent that minimizes the distance to other oscillators while incurring a quadratic cost. Notably, this game is not Lasry-Lions monotone, 
which precludes us from guaranteeing the uniqueness of Nash equilibria. Indeed, 
multiple solutions bifurcate from the incoherent state 
as the coupling strength, or the misalignment cost, increases.
This is the central parameter 
in these models causing the phase transition which can be summarized as follows:
\begin{itemize}
    \item \textit{Subcritical regime}: Below the critical coupling strength, oscillators behave incoherently and the uniform distribution is the unique stationary Nash equilibrium. Additionally, \cite{carmona_synchronization_2023} obtains a local stability result by proving that from any initial condition that is sufficiently close to the uniform distribution there exist time-inhomogeneous Nash equilibria converging to the uniform measure as time passes.
    
    \item \textit{Supercritical regime}: Above the critical interaction parameter, multiple non-uniform stationary Nash equilibria emerge, leading to partial self-organization.  As the interaction strength goes to infinity, these stationary Nash equilibria become more coherent and  converge to a Dirac measure representing full synchronization.  
\end{itemize}

\begin{wrapfigure}[16]{r}{0.35\textwidth}
    \centering
    \vspace{-1em}
    \includegraphics[width=.8\linewidth]{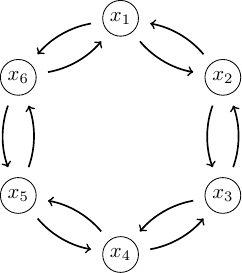}
    \caption{Example of a six-state model.}
    \label{fig:N-state-Kuramoto}
\end{wrapfigure}

Recently, Cesaroni \& Cirant \cite{cesaroni_stationary_2023} obtained further results with an ergodic cost functional in the \emph{large} parameter regime. They prove that for large misalignment parameters, there is only one non-uniform, self-organizing, stationary Nash equilibrium modulo translations. Furthermore, \cite{cesaroni_stationary_2023} shows that there are Nash equilibria of the finite horizon problem that converge to the self-organizing equilibrium.

In the Kuramoto model, the one-dimensional torus given by $\T=\R/(2\pi\Z)$ represents the set of possible phases of the oscillators.  We simplify this by considering an equidistant discretization $\cX_N=\{x_1,..,x_N\}\subset\T$ of the phase space, as done more generally by Bertucci \& Cecchin \cite{BC}.  We assume that the oscillators are subject to non-zero \textit{thermal noise} of strength $\sigma^2>0$ and without any control, they would move to their left or right on $\cX_N$ with a rate of $N^2\sigma^2$. In the game setting, oscillators are allowed to influence this transition by choosing a feedback control $\alpha := (\alpha_{\ell},\alpha_r): [0,\infty)\times\cX_N\mapsto[0,\infty)^2$. Then, the position of a generic oscillator is described by a continuous-time Markov chain $X^{\alpha}_t$ with controlled transition rates $\lambda^\alpha_{i,j}(t)$ from state $x_i$ to $x_j$ given by,
\begin{equation*}
    \lambda^{\alpha}_{i,j}(t) = 
    \begin{cases}
        N\alpha_{\ell}(t,x_i)  + N^2\sigma^2 & \text{if } j=i-1 \,\, (\text{mod}\,N),\\
        N\alpha_{r}(t,x_i) + N^2\sigma^2 & \text{if } j=i+1 \,\, (\text{mod}\,N),\\
        0 & \text{else.}
    \end{cases}
\end{equation*}
Figure \ref{fig:N-state-Kuramoto} is a visualization of a six-state model. 

We now describe the \emph{discretized Kuramoto MFG} which is a game between infinitely many such oscillators while the number of possible phases $N$ is kept finite. Let the probability flow $(\mu_t)_{t\geq0}$ on $\cX_N$  represent the population distribution of oscillators' phases. Then, $\mu_t$ solves a forward Kolmogorov equation which the representative oscillator or agent cannot influence. Instead, they try to align their own phase with $\mu_t$ while incurring a quadratic cost of effort by minimizing the corresponding discounted infinite cost,
\begin{equation}
\label{eq:J_beta}
  J_\beta(\alpha):=  \E \int_0^{\infty} e^{-\beta t} \left(\frac{1}{2}|\alpha(t,X^{\alpha}_t)|^2 + \kappa\, \ell(X^{\alpha}_t,\mu_t)\right)\,\d t,
\end{equation}
over feedback controls $\alpha$, where $X^\alpha_t\in\cX_N$ is the random position of the representative oscillator with controlled rates $(\lambda^{\alpha}_{i,j})$ and $X^{\alpha}_0\sim\mu_0$, the discount factor is $\beta>0$, the running cost $\ell$ is defined by
\begin{equation} \label{eq:Kuramoto_cost}
    \ell(x,\mu) := 2 \sum_{y \in \cX_N} \sin^2\left(\frac{x-y}{2}\right)\,\mu(\{y\}),
\end{equation}
and $\kappa>0$ is the strength of the interactions between oscillators. The second part of the running cost
encourages alignment with the population phase and, hence, is not Lasry-Lions monotone. We emphasize again that 
$\mu_t$ is fixed for the representative agent, and any optimal control $\alpha^\ast$ and its distribution 
$\cL(X^{\alpha^\ast}_t)$ are  functions of $\mu_t$. Hence, this construction defines a map 
$(\mu_t)_{t\geq0}\mapsto(\cL(X^{\alpha^\ast}_t))_{t\geq0}$, and
the fixed points of this map are the \textit{mean-field game (Nash) equilibria}. If a Nash equilibrium $t\mapsto\mu^*_t$ 
is constant, we call it a \emph{stationary MFG (Nash) equilibrium}. 
We note that in this case, the initial condition 
$\mu_0^*$ is no longer given but part of the solution. We say that a stationary Nash equilibrium 
$\mu^*$ is \textit{stable} if all time-inhomogeneous Nash equilibria 
$(\mu^*_t)_{t\geq0}$ converge towards $\mu^*$ as $t\to\infty$. 

\subsection{Related Studies}

This paper studies the emergence and properties of self-organizing equilibria as the outcome of a game among a continuum of rational agents that favor alignment with the majority.  As such it is situated at the intersection of several fields.  In addition to several already mentioned studies on Kuramoto MFG and synchronization, our model is related to the studies of opinion dynamics as studied in \cite{acemoglu_opinion_2013, bauso_opinion_2016,castellano_statistical_2009} and the references therein.  The static mean-field game \emph{``Where do I put my towel on the beach?''} that Lions discusses in \cite{lions} is also similar to our model. In such a setting, there is non-uniqueness once agents favor crowds as opposed to avoiding them.  

Our model bears resemblance to \emph{coordination games} in game theory where non-cooperative agents wish to independently play the same strategies and in which, as a result, multiple equilibria emerge. It is also connected to the study of opinion dynamics and consensus formation, see Dai Pra, Sartori, \& Tolotti \cite{dai_pra_climb_2019} who study the emergence of collective behavior and investigate a dynamic model that is similar to ours in the special case of $\sigma=0$. The absence of thermal noise, however, precludes the emergence of the phase transition and random structures. 

Furthermore, our work contributes to a growing literature of finite-state mean-field games \cite[Chapter 7.2]{carmona_probabilistic_2018-1}.  
These include \cite{kolokoltsov_mean-field-game_2016,kolokoltsov_corruption_2018,kolokoltsov_mean-field-game_2017} 
who study the spread of corruption and a botnet defense model. The first two papers analyze a game with three and four states in a bang-bang type model that exhibits a multiplicity of stationary equilibria. Socio-economic applications of finite-state MFGs, including a description of the potential game structure, can be found in \cite{gomes_socio-economic_2014}. Classical results establishing existence, uniqueness under the Lasry-Lions monotonicity condition, and the convergence of the finite player games in finite-state MFGs can for example be found in \cite{carmona_probabilistic_2021,cecchin_probabilistic_2020,gueant_existence_2015}. Cohen \& Zell \cite{cohen_analysis_2023} treat the infinite horizon case, and \cite{cecchin_convergence_2019} studies the convergence problem of a two-state MFG under an anti-monotonous cost. 
In recent work, \cite{cohen2024asymptotic} prove that any of the time-dependent ergodic Nash equilibria presented in our paper give rise to a $(C/\sqrt{n})$-Nash equilibrium in the $n$-player game, see their Remark 3.4.

The paper is organized as follows:
We formulate the problem and define
the discounted and ergodic Nash equilibria in 
Section~\ref{sec:2MFG}, and we summarize the 
results in Section~\ref{sec:Results}.
We characterize the Nash equilibria as 
bounded solutions of a coupled system 
of ordinary differential equations in
Section~\ref{sec:characterization}
and discuss the stationary ones in 
Section~\ref{sec:Stationary_Equilibria}. 
The analysis of the discounted problem is carried
out in Section~\ref{sec:discounted}
and the ergodic one in Section~\ref{sec:ergodic}.
Extensions and conclusions are given in 
Section~\ref{sec:eg}, and 
we prove two technical results in the Appendix.

\section{Two-State Mean-Field Game} 
\label{sec:2MFG}

In what follows, we study the two-state game. We replace the state space $\cX_2$ 
with the discrete set $\mathcal{X}:=\{0,1\}$, and a probability measure $\mu$ on $\cX$ is identified with its value $p=\mu(\{1\})$. For $x\in\cX$, $x+1$ denotes summation mod $2$.  
In our model agents are incentivized to move
themselves towards the majority by minimizing 
the \emph{running cost} $\ell$ given by,
\begin{equation}\label{eq:cost}
\ell(x,p) := 
    \begin{cases}
        p & \text{if } x = 0,\\
        1-p & \text{if } x = 1,
    \end{cases}
    \qquad x \in \cX, \ p \in [0,1].
\end{equation}
Whenever the majority of the population is in state $1$, or equivalently when $p>1/2$, 
the above  running cost satisfies $\ell(0,p)>\ell(1,p)$,  encouraging the agents to align themselves with the majority 
of the players who are in state $1$. 
 In the opposite case when $p<1/2$,
$\ell(1,p) >\ell(0,p)$ making the majority state $0$
favorable.
The important  \emph{coupling constant} $\kappa$
that multiplies $\ell$ in \eqref{eq:J_beta}
is the \emph{misalignment strength}.
We further observe that, up to scaling, this is exactly the 
running cost introduced earlier in the $N$-state game, and it is not Lasry-Lions monotone.

\subsection{State Process}
\label{ssec:state}

\begin{Def}
\label{def:admissible}
{\rm{The set of}} feedback controls {\rm{$\cA$ consists of functions 
$\alpha:[0,\infty)\times\cX\mapsto[0,\infty)$ such that $\alpha(\cdot,x)$ is locally integrable.
We call a feedback control $\alpha \in \cA$}} special 
{\rm{if $\alpha(t,0)\, \alpha(t,1)=0$ holds for all times $t\geq0$.
$\cA_s \subset \cA$ denotes the  collection of all special
feedback controls.}}
\end{Def}
For $\alpha \in \cA$, we let $\sigma^2+ \alpha(t,0)$ be the transition rate from state $0$ to $1$ and 
$\sigma^2+ \alpha(t,1)$ is the rate from $1$ to $0$.
Here, $\sigma^2>0$ is a fixed \emph{thermal noise}. 
Following the interpretation discussed in the Introduction, 
for any $\alpha \in \cA$ and any initial distribution there exists a Markov chain $X^{\alpha}_t$ 
with the transition-rate matrix
\begin{equation*}
    Q^{\alpha}_t= 
    \begin{pmatrix}
        -\sigma^2 - \alpha(t,0) & \sigma^2 + \alpha(t,0)\\
        \sigma^2 + \alpha(t,1) & -\sigma^2 - \alpha(t,1)
    \end{pmatrix},
\end{equation*} 
where $\sigma^2>0$ is the strength of the thermal noise. 
Then, $p(t) = \P(X^\alpha_t=1)$ solves the Kolmogorov equation,
\begin{equation}
\label{eq:KOL}
\Dot p(t) = (\sigma^2+\alpha(t,0))(1-p(t))- (\sigma^2+\alpha(t,1))p(t).
\end{equation}
The following straightforward result proved in Appendix \ref{app:equal}
shows that there is no loss of generality in restricting the feedback controls 
to be special.

\begin{Lem}
\label{lem:equal}
For any $\alpha \in \cA$ there exists $\tilde \alpha \in \cA_s$
such that for every $x\in \cX$ and $ t \ge 0$,
$\P(X^{\tilde \alpha}_t=1) =\P(X^{\alpha}_t=1)$ and
$|\tilde \alpha(t,x)| \le  | \alpha(t,x)|$. Moreover,
for any probability flow $p$
there exists a unique $\alpha_p \in \cA_s$ such that
$p$ solves the Kolmogorov equation \eqref{eq:KOL}
with $\alpha_p$.
\end{Lem}

For any $\alpha\in \cA_s$,  there is a unique  measurable function of time 
defined by $a_\alpha(t):= \alpha(t,0)-\alpha(t,1)$.
Conversely for any measurable function $a :[0,\infty) \mapsto \R$, there exists
a canonical special feedback control $\alpha^a \in \cA_s$
given by $\alpha^a(t,0):=a^+(t)$,
$\alpha^a(t,1):=a^-(t)$,  where for
a real number $z$, $z^+=\max\{z,0\}$ is its positive part of $z$,
and $z^-=\max\{-z,0\}$ is its negative part. 
It is clear that $a(t)=a_{\alpha^a}(t)=\alpha^a(t,0)-\alpha^a(t,1)$,
and these maps provide a one-to-one
connection between special feedback controls
and measurable functions of time. Further, for a function $a=a(t)$, we write $X^a=X^{\alpha^a} $ for the Markov chain controlled by $\alpha^a$ and with an abuse of terminology, call $a$ a feedback control
without referring to $\alpha^a$ explicitly.

\subsection{Discounted Nash Equilibria}
\label{ssec:dNE}
Recall the cost functional $J_\beta$ of \eqref{eq:J_beta} with discount factor $\beta>0$ and specialized to the two-state setting. In particular, the running cost $\ell$ is now given by \eqref{eq:cost}.

\begin{Def}
\label{def.mfgbeta}
{\rm{We say that a flow of probabilities $t\mapsto p(t) \in [0,1]$ is a \emph{discounted mean-field game Nash equilibrium} starting from $p(0)$, if there exists $\alpha^*\in\mathcal{A}$ such that: }}
\begin{enumerate}
    \item {\rm{$\alpha^*$ minimizes the discounted cost functional $J_\beta(\alpha)$ over $\alpha\in\cA$, in which $X^{\alpha}_t\in\cX$ is the Markov chain 
    with $\P(X^{\alpha}_0=1)=p(0)$ and
    transition matrix $(Q^{\alpha}_t)_{t\geq0}$.}}
    \item {\rm{For all times $t\geq0$, we have $p(t)=\P(X^{\alpha^*}_t=1)$.}}
\end{enumerate}
\end{Def}

\begin{Rmk}
\label{rem:why}
{\rm{In view Lemma \ref{lem:equal},
for any $\alpha\in\cA$ there is $\tilde \alpha \in \cA_s$
such that $J_\beta(\tilde \alpha) \le J_\beta(\alpha)$.
Thus, we may assume that $\alpha^* \in \cA_s$ in the above 
definition.  Again by Lemma \ref{lem:equal},
as $p$ solves the Kolmogorov
equation \eqref{eq:KOL} with $\alpha^*$, we conclude
that $\alpha^*=\alpha_p$
Therefore, in equilibrium $p$ completely determines
the optimal strategy, and this is the main reason why we 
define Nash equilibria only through $p$.}}
\end{Rmk}

\subsection{Ergodic Nash Equilibria}
\label{ssec:eNE}

For a given flow of probabilities $t\mapsto p(t) \in [0,1]$ we define the \emph{ergodic cost functional} $J_0$ by,
$$
J_0(\alpha):= \underset{T\uparrow\infty}{\overline{\lim}}\ \frac1T \
\E\int_0^T\left(\frac12 \alpha(t,X^\alpha_t)^2 + \kappa \ell(X^{\alpha}_t,p(t))\right)\,\d t,\quad \alpha\in\cA,
$$
where as before $X^\alpha_t$ denotes the Markov chain controlled by
the feedback control $\alpha$ and $\P(X^\alpha_0=1)=p(0)$. 
Clearly, $J_0$ is invariant under changes of the control 
and the probabilities on bounded time intervals, and only the tail behavior matters. 
Thus,  any meaningful definition of Nash equilibria (NE)
must restrict either the controls or the probabilities.

For instance, as a first attempt, suppose that we define ergodic NE by simply 
replacing $J_\beta$ with $J_0$ in the previous definition,
and  assume that $p(t)\equiv p(0)$ is a stationary NE with optimal control 
$\alpha^*$. By modifying the probability flow  and the control on a finite time
interval, we could construct infinitely many other 
probabilities $\tilde p$ with controls
$\tilde a^*$ such that $\tilde a^*$  minimizes
$J_0$ with the given $\tilde p$, and $\tilde p(t)= \P(X^{\tilde a^*}_t = 1)$.
Consequently, all of these constructions would be considered as 
distinct NE under this definition,
despite sharing the same tail behavior.
To avoid this redundancy, we restrict NE to
be periodic or stationary, thereby implicitly consider them
as the natural representative of the other equilibria with the same
tail behavior.  In particular,
this restriction identifies all the flows $\tilde p$ that we constructed above
with the original stationary one.
\begin{Def}
{\rm{We say that a \emph{periodic} or a
\emph{stationary} flow of probabilities 
$t\mapsto p(t) \in [0,1]$ is an \emph{ergodic mean-field game Nash equilibrium} if it 
there exists a control  $\alpha^*\in\mathcal{A}$ such that
$\alpha^*$ minimizes the ergodic cost functional $J_0$ and 
$p(t)=\P(X^{\alpha^*}_t=1)$ for all times $t \ge 0$.}}
\end{Def}

The periodicity restriction in the above definition allows us to obtain 
the desirable characterization of NE as solutions of the MFG system
(\ref{eq:Kol},\ref{eq:HJ_control}). Moreover, it does not cause any loss 
of generality in our model as it is well-known that 
the limit behavior of such planar systems is always periodic.

\begin{Rmk} 
\label{rem:why_erg} {\rm{In view of the discussion given above, the optimal control is not unique in the general class of feedback controls $\cA$.
However, by Lemma \ref{lem:equal} once there is a minimizer $\alpha^*$, 
we may assume that $\alpha^*$ is a special feedback control.  
In fact, if we assume that the minimizer in the above definition is in $\cA_s$, 
then it must be equal to $\alpha^*=\alpha_p$.
In particular, $\alpha_p$ is also periodic with the same period as $p$.}} 
\end{Rmk}

\section{Results}
\label{sec:Results}

We study the two-state model, providing a complete characterization of the set of  time-inhomogeneous 
and stationary equilibria and their stability properties for all values of the parameters $(\beta,\sigma^2,\kappa)$. In addition to the discounted cost functional, we also investigate the model with an ergodic cost which formally corresponds to the limit $\beta\downarrow0$, and we further identify the equilibria that are obtained as first-order conditions of the associated mean-field control problem, see Sections \ref{ss:potential} and \ref{ss:ergodic_potential}.

By standard techniques from optimal control we show in subsection \ref{ss:cov} that given a flow of probabilities $p(\cdot)$ the optimal feedback control $\alpha$ is unique for the discounted problem and is given by,
$$
\alpha(t,0)= (v(t,0)- v(t,1))^+, \qquad
\alpha(t,1)= (v(t,0)- v(t,1))^-,
$$
where $v$ is the value function defined in \eqref{eq:value}. 
In view of subsection \ref{ssec:state},  the above
optimal control $\alpha=\alpha^a$ with $a(t):= v(t,0)-v(t,1)$. If $p(t)$ is a discounted Nash equilibrium, then $p(t) = \P(X^a_t=1)$ 
and it solves forward Kolmogorov equation,
\begin{equation}
    \label{eq:Kol}
    \Dot p(t) = (\sigma^2+a^+(t))(1-p(t))- (\sigma^2+a^-(t))p(t).
\end{equation}
Additionally, using the dynamic programming equation satisfied
by the value $v$, we show that $a$ solves,
\begin{equation}
\Dot a(t) = (\beta + 2\sigma^2) a(t) + \frac{1}{2}\mathrm{sign}(a(t))a(t)^2 - \kappa (2p(t)-1) \label{eq:HJ_control},
\end{equation}
where $\beta=0$ corresponds to the ergodic cost. 

Combining these, we characterize all discounted and ergodic Nash equilibria as bounded solutions of (\ref{eq:Kol},\ref{eq:HJ_control}). This one-to-one connection between the Nash equilibria and (\ref{eq:Kol},\ref{eq:HJ_control}) is established in Propositions \ref{pr:ergodic} and \ref{pr:infinite}.
We continue by summarizing its consequences, while precise statements and their proofs are provided in the following sections.

\subsection{Discounted Cost}

\begin{wrapfigure}[22]{r}{0.5\textwidth}
    \vspace{-1em}
    \centering
    \includegraphics[width=.73\linewidth]{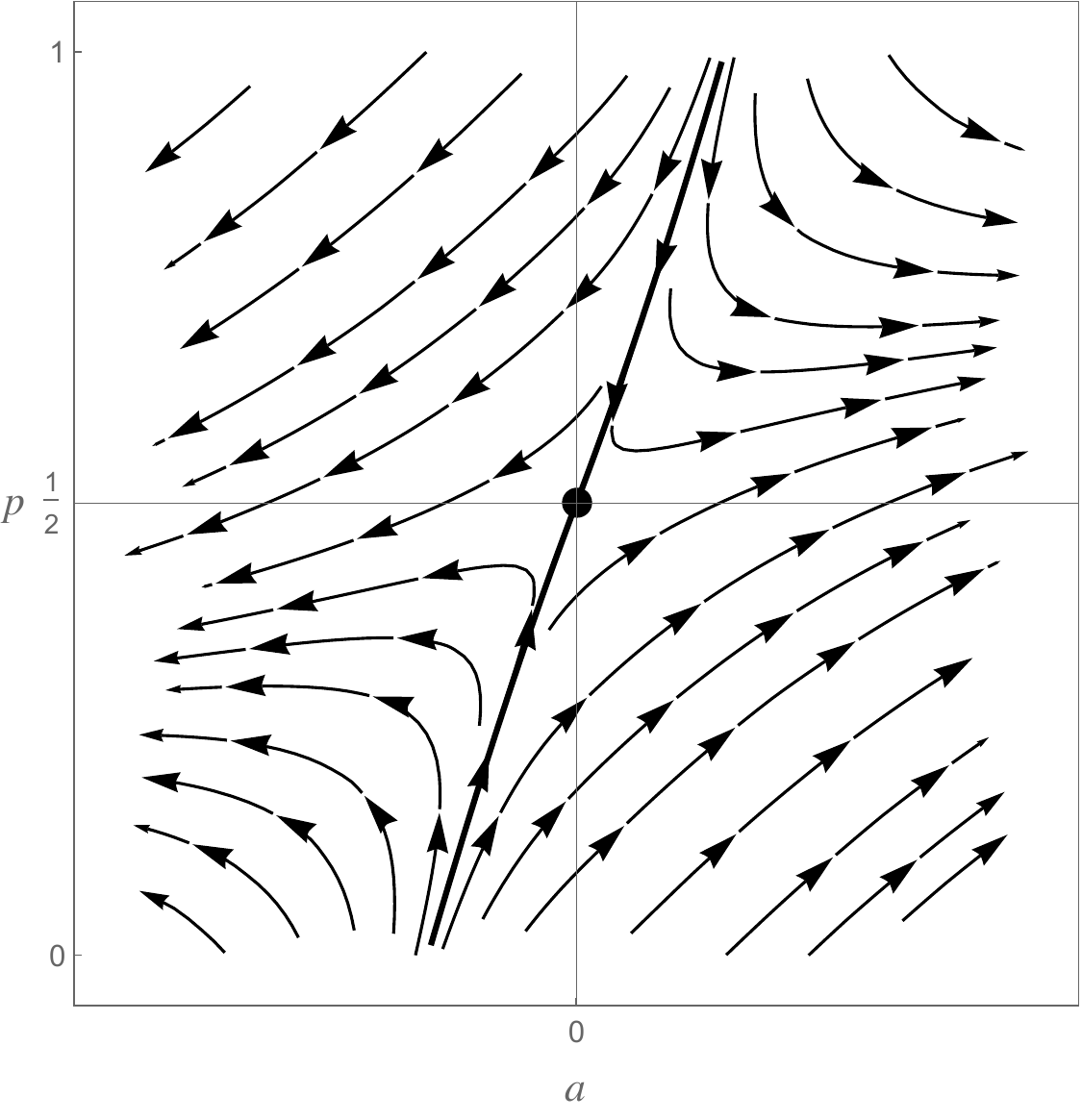}
    \caption{Phase diagram of the discounted subcritical system (\ref{eq:Kol},\ref{eq:HJ_control}). The dot shows the location of the stationary Nash equilibrium while the thick lines illustrate time-inhomogeneous Nash equilibria.}
    \label{fig:Phase_Diag_disc_sub}
\end{wrapfigure}

For a positive discount factor $\beta>0$ the critical coupling or misalignment  strength is
\begin{equation*}
    \kappa_c = 2\beta \sigma^2 + 4\sigma^4.
\end{equation*}

\noindent
\textit{Subcritical regime $\kappa<\kappa_c$}. The uniform distribution is the unique stationary Nash equilibrium (SNE) and it is stable.  Further, for any initial distribution there exists a unique time-inhomogeneous Nash equilibrium (NE), and it converges to the uniform distribution.

\noindent
\textit{Supercritical regime $\kappa>\kappa_c$}. There exist three SNE: the uniform distribution and two symmetric self-organizing SNE given by $\underline{p}<1/2<\overline{p}$ where $\underline{p}=1-\overline{p}$. The local behavior around the uniform distribution depends on two sub-regimes:
\begin{enumerate}[(A)]
    \item $\kappa_c<\kappa<\kappa_c+\beta^2/4$. For any non-uniform initial distribution  there exists a unique NE and it converges to one of the self-organizing SNE. 
    
    \item $\kappa>\kappa_c+\beta^2/4$. For initial conditions close to the uniform distribution there exist many NE that spiral around the uniform distribution before converging to one of the self-organizing SNE.
\end{enumerate}

We visualize these results by the phase diagrams of the ordinary differential equations (\ref{eq:Kol},\ref{eq:HJ_control}). 
The uniform SNE with zero control corresponds to the fixed point $(a,p)=(0,1/2)$ of this system of equations.
The local behavior of the solutions around this point changes with increasing interaction and can be summarized by the eigenvalues $\lambda_1,\lambda_2$ of the linearized system around $(0,1/2)$.

In the subcritical regime, the point $(0,1/2)$ is a saddle, i.e.\ $\lambda_1<0<\lambda_2$, and its stable manifold crosses the boundary $\{p=0,1\}$ as seen in Figure \ref{fig:Phase_Diag_disc_sub}.

When $\kappa>\kappa_c$, the two self-organizing SNE are saddles and their stable manifolds join at the origin, creating a curve $\cC$ that connects all SNE and hits the boundaries $\{p=0,1\}$. In the weakly supercritical regime (A), the curve $\cC$ is monotone and the uniform SNE becomes repellent with both eigenvalues being positive, as shown in Figure \ref{fig:Phase_Diag_disc_supA}. In the supercritical regime (B), $\lambda_1, \lambda_2$ become complex with positive real parts, so that $(0,1/2)$ is a spiral source, see Figure \ref{fig:Phase_Diag_disc_supB}. In all phase diagram, the thick dots show the location of stationary equilibria.

\begin{figure}[h]
    \begin{subfigure}{0.45\textwidth}
        \centering
        \includegraphics[width=.79\linewidth]{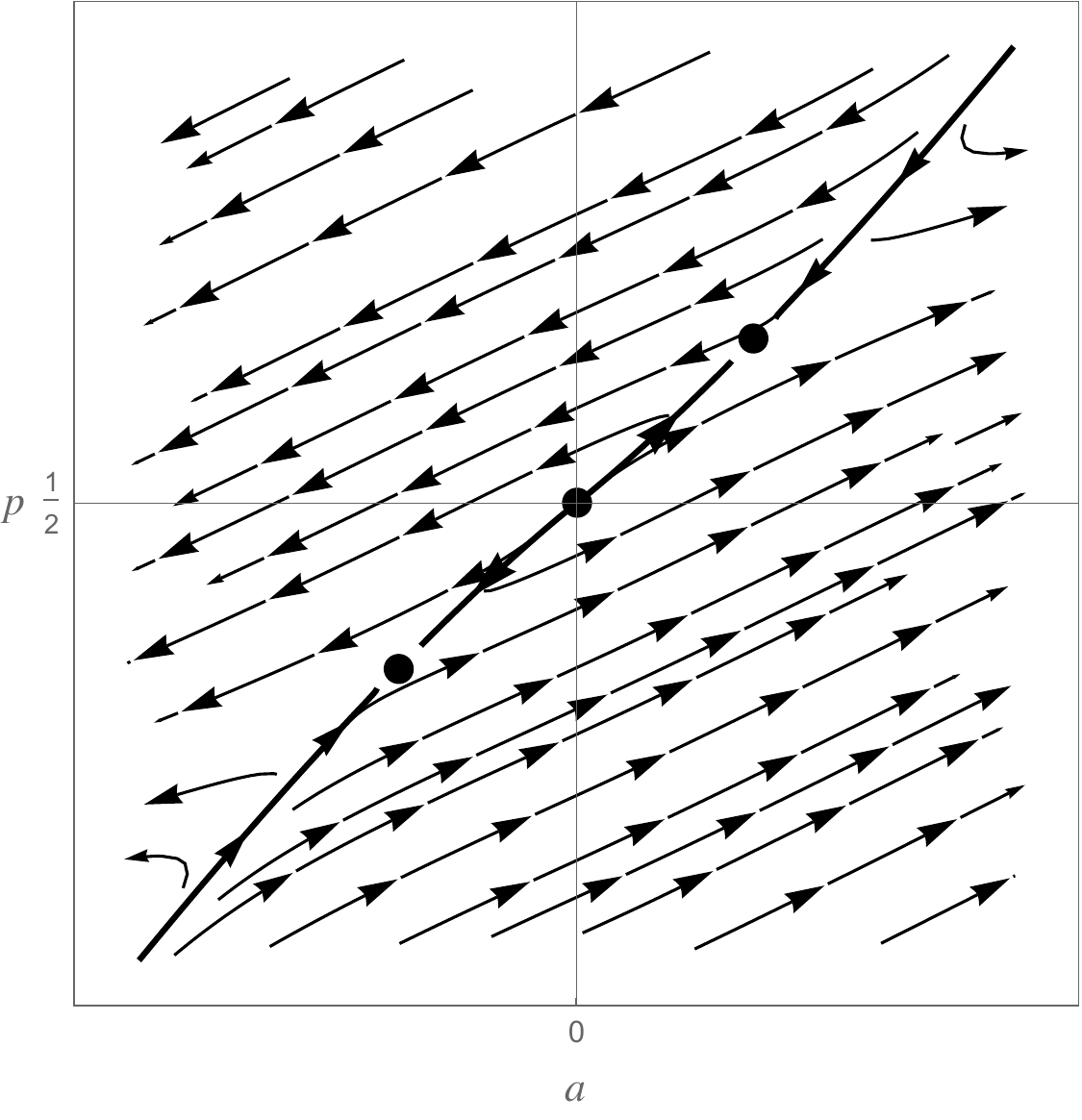}
        \caption{Supercritical case (A): $\kappa_c<\kappa<\kappa_c+\beta^2/4$. The thick line corresponds to the monotone curve $\cC$, and the dots show the stationary equilibria.}
        \label{fig:Phase_Diag_disc_supA}
    \end{subfigure}\hspace{5pt}
    \begin{subfigure}{0.45\textwidth}
        \centering
        \includegraphics[width=.79\linewidth]{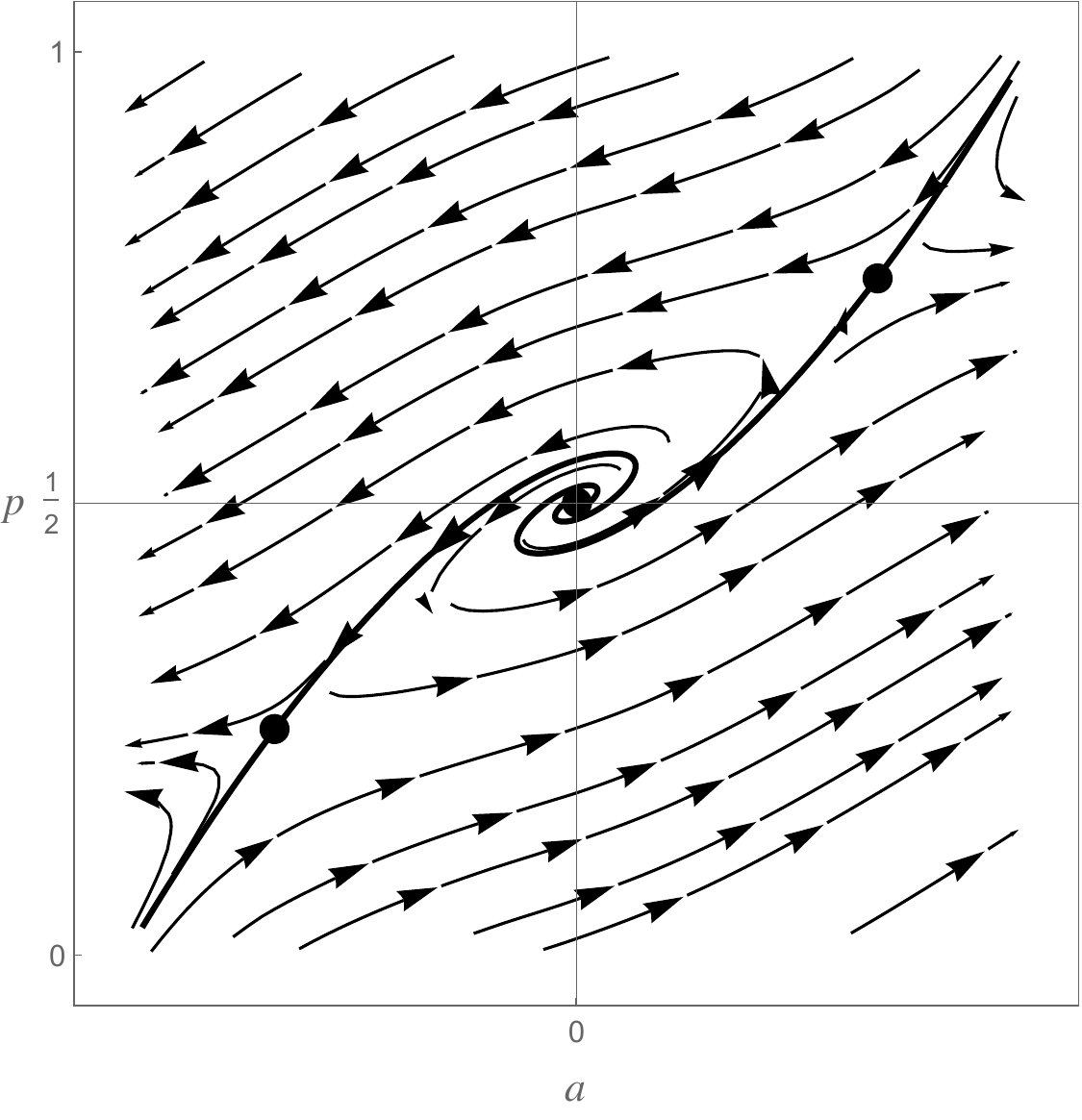}
        \caption{Supercritical case (B): $\kappa>\kappa_c+\beta^2/4$. The thick line again corresponds to the curve $\cC$, which spirals around $(0,1/2)$.}
        \label{fig:Phase_Diag_disc_supB}
    \end{subfigure}
    \caption{Phase diagrams of the supercritical discounted system (\ref{eq:Kol},\ref{eq:HJ_control})}
\end{figure}

All Kuramoto MFGs are potential and there is an associated mean-field control (MFC) problem introduced in subsections \ref{ss:potential}, \ref{ss:ergodic_potential} below. The minimizers of these problems are Nash equilibria for the MFG. Surprisingly, with a discounted cost, the uniform distribution is selected as the minimizer  in parts of the supercritical regime.

\subsection{Ergodic Cost}

\begin{wrapfigure}{r}{0.5\textwidth}
    \vspace{-2em}
    \centering
    \includegraphics[width=.73\linewidth]{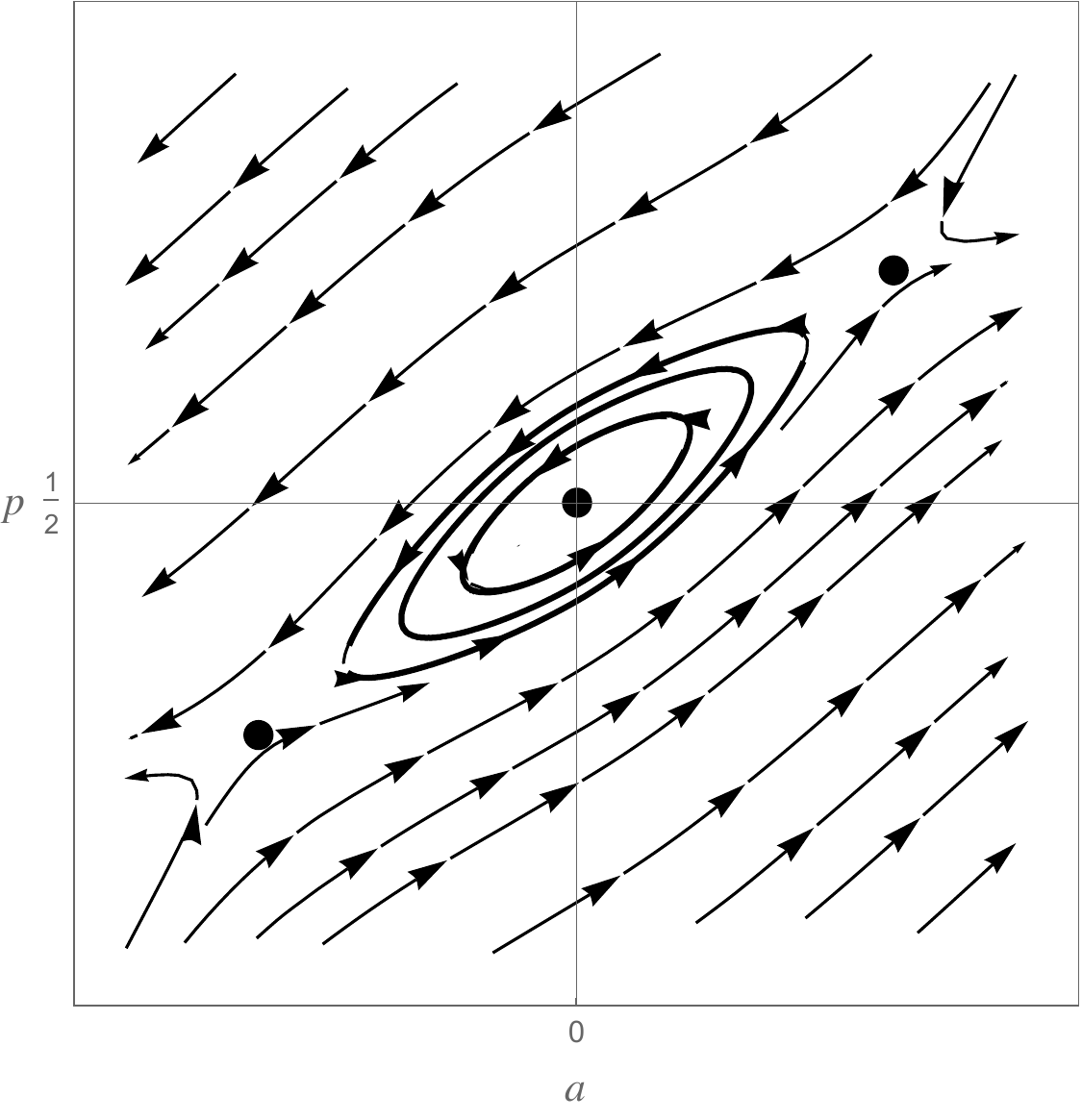}
    \caption{Phase diagram of the supercritical ergodic system (\ref{eq:Kol},\ref{eq:HJ_control}).}
    \label{fig:Phase_Diag_erg}
\end{wrapfigure}

Here the critical interaction parameter is
\begin{equation*}
    \kappa_c=4\sigma^4.
\end{equation*}

In the subcritical regime, the uniform distribution $p^*=1/2$ is again the unique SNE, and there are no other ergodic NE. 

In the supercritical case $\kappa>\kappa_c$, in addition to the uniform distribution, there are two other symmetric SNE. However, in contrast to the discounted model, there are infinitely many periodic NE rotating around the uniform distribution as well. Figure \ref{fig:Phase_Diag_erg} shows the phase diagram of the system (\ref{eq:Kol},\ref{eq:HJ_control}) close to the uniform distribution, and in dynamical systems terminology, the origin is a center of the equations (\ref{eq:Kol},\ref{eq:HJ_control}).

In the subcritical regime, the uniform distribution is the minimizer of the associated ergodic mean-field control (MFC) problem. It fails to be the minimizer in the supercritical case.

\section{Characterization of Nash Equilibria} 
\label{sec:characterization}

As described in the previous sections, Nash equilibria are in a natural one-to-one correspondence with a system of forward-backward ordinary differential equations which we refer to as the \emph{MFG system}. To state this characterization, for $x \in \cX$, $v \in \R$, and $p \in [0,1]$, we define a \emph{Hamiltonian} by
\begin{equation}
    H(x,a,p) := \inf_{z\geq0} \left\{ (\sigma^2+z)a + \frac{1}{2} z^2 \right\} + \kappa\,\ell(x,p) 
    = \sigma^2a -\frac12 (a^-)^2 + \kappa \ell(x,p). 
    \label{eq:Hamiltonian}
\end{equation}

\subsection{Ergodic MFG System}
The ergodic dynamic programming equation is
\begin{equation}
    \label{eq:HJ_erg}
    - \partial_t v(t,x) + \overline{\lambda} = H(x,v(t,x+1)-v(t,x),p(t)),
\end{equation}
for every $x \in \cX, t \ge  0$,  and an optimal control is given by
\begin{equation}
    a(t)=  v(t,0)-v(t,1).
\label{eq:opt_feedback}
\end{equation}
The unknowns in \eqref{eq:HJ_erg} are the constant $\overline{\lambda}$ and the function $v(t,x)$. Classically, given a flow $p(\cdot)$, such a solution pair $( \overline{\lambda} ,v)$  is constructed by letting $\beta$ to zero in the corresponding discounted control problem. We detail this argument in Appendix \ref{app:erg} since the time-dependence in the ergodic dynamic programming equation \eqref{eq:HJ_erg} is non-standard.

When the flow $p(\cdot)$ is a Nash equilibrium, then $p(\cdot)$ solves \eqref{eq:Kol} with the feedback control defined above, and the ergodic mean-field game system $(\mathrm{MFG}_0)$ consists of the equations \eqref{eq:HJ_erg}, \eqref{eq:Kol} coupled by \eqref{eq:opt_feedback}. Precisely, we call a triplet $(\overline{\lambda},v,p)$ a \textit{classical solution} of $(\mathrm{MFG}_0)$ if $\overline{\lambda}\in\R$ is a real number, and $v:[0,\infty)\times\mathcal{X}\mapsto\R$, $p:[0,\infty)\mapsto[0,1]$ are continuously differentiable in the time variable, and they satisfy \eqref{eq:HJ_erg} and \eqref{eq:Kol} with $a$ given by \eqref{eq:opt_feedback}. 

We proceed to establish the correspondence of NE with solutions of $(\mathrm{MFG}_0)$, and we also show that $\overline{\lambda}$ has an interpretation as the optimal value of the ergodic cost,
\begin{equation} 
    \overline{\lambda} = \inf_{\alpha\in\mathcal{A}} \ \ \underset{T\uparrow\infty} {\overline{\lim}}\ \ \frac{1}{T} \ \E\int_0^T\left(\frac{1}{2}\alpha(t,X^\alpha_t)^2 + \kappa \ell(X^{\alpha}_t,p(t))\right)\,\d t.
    \label{eq:erg_value}
\end{equation}
Let $\cA_c$ be the set of feedback controls in $\cA$ that are continuous in the time variable.

\begin{Prop}
\label{pr:ergodic}
For a continuous periodic flow of probabilities $p(\cdot)$, the following are equivalent:
\begin{enumerate}[(i)]
    \item 
    \label{i} $p(\cdot)$ is an ergodic Nash equilibrium.
    \item 
    \label{ii}
    There are a constant $\overline{\lambda}$ and $v:[0,\infty)\times \cX \mapsto \R$ such that $(\overline{\lambda},v,p)$ is a bounded classical solution  of $(\mathrm{MFG}_0)$.
    \item 
    \label{iii}
    There is a feedback control $a\in \cA_c$ such that $(a,p)$ is a periodic solution of (\ref{eq:Kol},\ref{eq:HJ_control}) with $\beta=0$.
\end{enumerate}
In all cases, $\overline{\lambda}$ is given by \eqref{eq:erg_value}.
\end{Prop}
\begin{proof}
    $(\ref{ii}) \Rightarrow (\ref{i})$. Since $v(\cdot,x)$ is continuously differentiable, Dynkin's formula implies, see \cite[Appendix B]{fleming_soner}, that for any feedback control $\alpha\in\mathcal{A}$, 
    \begin{equation*}
        \E[v(T,X^{\alpha}_T)] - \E[v(0,X^{\alpha}_0)] = \E \int_0^T 
        \left(\partial_t + \ig^\alpha\right)v(t,X^{\alpha}_t) \,\d t,
    \end{equation*}
    where $X^{\alpha}_0\sim p(0)$ and $\ig^{\alpha}$ is the infinitesimal 
    generator of the controlled Markov chain: for $x\in \cX$, $t\ge 0$,
    \begin{equation*}
        \ig^{\alpha}\Phi(t,x) =(\sigma^2+\alpha(t,x)) (\Phi(t,x+1)-\Phi(t,x)).
    \end{equation*}
    Since $(\overline\lambda,v)$ satisfy \eqref{eq:HJ_erg} and by definition of the Hamiltonian \eqref{eq:Hamiltonian}, we have
    \begin{equation*}
        \E \int_0^T \left(\partial_t + \ig^{\alpha}\right)v(t,X^{\alpha}_t) \,\d t 
        \geq T\overline{\lambda} - \E \int_0^T \left(\frac{1}{2}\alpha(t,X^\alpha_t)^2 
        + \kappa \ell(X^{\alpha}_t,p(t))\right) \,\d t.
    \end{equation*}
    Using the boundedness of $v$ we divide by $T$ and send $T\uparrow\infty$ to conclude that $\overline{\lambda}$ is a lower bound for the ergodic cost and that $\alpha^a$ is optimal where $a$ is given by \eqref{eq:opt_feedback}. Since $\P(X^a_t=1)$ satisfies the same forward Kolmogorov equation \eqref{eq:Kol} as $p(\cdot)$, by uniqueness  $p(t)=\P(X^a_t=1)$ for all $t\geq0$. Together with the periodicity condition, this proves that $p(\cdot)$ is an ergodic NE. 
    \vspace{4pt}

\noindent
    $(\ref{i}) \Rightarrow (\ref{ii})$. Suppose $p(t)$ is an ergodic NE. In view
    of Lemma \ref{lem:equal} and Remark  \ref{rem:why_erg}
    there is a corresponding 
    optimal control $\alpha^*\in\cA_s$ and it is given by $\alpha^*=\alpha_p$.
    The explicit formula for $\alpha_p$ in Appendix \ref{app:equal}
    shows that $p$ and $\alpha^*$ have a common period $\tau$.
    
    Let $(\overline{\lambda},v(t,x))$ be a bounded solution of the dynamic programming equation 
    \eqref{eq:HJ_erg} with $p(t)$, constructed as the vanishing discount limit as in Lemma \ref{lem:vanish-disc}, 
    and define $a$ by \eqref{eq:opt_feedback}. 
    We claim that $\alpha^*(t,0)=a(t)^+$ and $\alpha^*(t,1)=a(t)^-$ for Lebesgue a.e.~$t\geq0$. As $p(t)=\P(X^{\alpha^*}_t=1)$ solves the Kolmogorov equation \eqref{eq:Kol} with $\alpha^*(t)$ replacing $a(t)$, this claim would imply that the triplet  $(\overline{\lambda},v,p)$ is a bounded classical solution of $(\mathrm{MFG}_0)$. So it suffices to prove this claim.
    
   Let $X^*_t$ be the Markov chain starting in $X^*_0\sim p(0)$ and controlled by $\alpha^*$. 
   Then, as $p(\cdot)$ and $(\alpha^*(\cdot,0), \alpha^*(\cdot,1))$ are
   $\tau$-periodic, we have
    \begin{equation*}
        \underset{T\uparrow\infty}{\overline{\lim}} \ \frac1T \ 
        \E \int_0^T \left(\frac{1}{2}(\alpha^*(t,X^{\alpha^*}_t))^2 
        + \kappa \ell(X^*_t,p(t))\right) \,\d t = \frac{1}{\tau}
        \E \int_0^\tau \left(\frac{1}{2}(\alpha^*(t,X^{\alpha^*}_t))^2 + \kappa \ell(X^*_t,p(t))\right) \,\d t.
    \end{equation*}
    Towards a contradiction, assume there is $x'\in\cX$, $\delta>0$, 
    and a subset $I\subset[0,\tau)$ of positive Lebesgue measure satisfying
    \begin{equation}\label{eq:contra}
    \begin{split}
        &(\sigma^2+\alpha^*(t,x')) (v(t,1+x')-v(t,x')) + \frac{1}{2}(\alpha^*(t,x'))^2 + \kappa\ell(x',p(t)) - \delta\\ 
        &\qquad \geq H(x',v(t,1+x')-v(t,x'),p(t)),\qquad \forall t \in I.
    \end{split}
    \end{equation}
    By periodicity, equation \eqref{eq:HJ_erg}, and the Dynkin's formula, 
    \begin{align*}
        0&= \E[v(\tau,X^*_\tau)] - \E[v(0,X^*_0)]
        = \E \int_0^\tau (\partial_t + \ig^{\alpha^*})\, v(t,X^*_t) \,\d t\\
        &\geq \tau \overline{\lambda} - \E \int_0^\tau \left(\frac{1}{2}(\alpha^*(t,X^{\alpha^*}_t))^2 
        + \kappa \ell(X^*_t,p(t))\right) \,\d t + \delta \int_{ I} \P(X^*_t=x')\,\d t.
    \end{align*}
    Hence, $\tau \overline{\lambda} +  \delta \int_I \P(X^*_t=x')\,\d t \ \leq \ \tau \overline{\lambda}$. This is impossible, as $\sigma>0$ precludes the Markov chain $X^*$ to be identically equal to one state. This proves the claim.
    \vspace{4pt}

\noindent
    $(\ref{iii}) \Rightarrow (\ref{ii})$. Let $(a,p)$ be a  continuous 
    periodic solution of (\ref{eq:Kol},\ref{eq:HJ_control}) with period $\tau>0$.
    For any $\lambda \in \R$, and $x \in \cX$, $t \in [0,\tau]$ set
    \begin{align*}
    V(t,0):= a(0)+\int_0^t [\lambda- H(0,-a(s),p(s))]\ \d s,\quad 
    V(t,1):= \int_0^t [\lambda- H(1,a(s),p(s))]\ \d s,
    \end{align*}
    and $A(t):= V(t,0)-V(t,1)$. We now use the fact that $a$ satisfies \eqref{eq:HJ_control} and the explicit form of $H$, to compute
    \begin{align*}
    \Dot A(t)&=H(1,a(t),p(t)) -H(0,-a(t),p(t))
    = 2\sigma^2 a(t) -\frac12(a(t)^-)^2 +\frac12 ((-a(t))^-)^2 - \kappa(2p(t)-1)\\
    &= 2\sigma^2 a(t) +\frac12 \mathrm{sign}(a(t))a(t)^2- \kappa(2p(t)-1)
    = \Dot{a}(t).
    \end{align*}
    As $A(0)=a(0)$, $A\equiv a$.  Therefore, $V$ solves the dynamic programming equation \eqref{eq:HJ_erg} with $p(\cdot)$.  Since by hypothesis $p(\cdot)$ solves \eqref{eq:Kol} with $a$, the triplet $(\lambda,V,p)$ is a classical solution of $(\mathrm{MFG}_0)$ for any $\lambda$.  To obtain a bounded solution, we choose
    $$
    \lambda =\frac{1}{\tau} \int_0^\tau  H(1,a(s),p(s)) \d s.
    $$
    Since $a$ and $p$ are $\tau$-periodic, the above choice of $\lambda$ ensures that $V(\cdot,1)$ is also $\tau$-periodic. Since $V(\cdot,0) = V(\cdot,1) + A(\cdot)$, it is also periodic. Therefore, $V$ is periodic, and hence bounded.
\vspace{4pt}

\noindent
    $(\ref{ii}) \Rightarrow (\ref{iii})$.  Let $(\overline{\lambda},v,p)$ be a bounded, classical solution of $(\mathrm{MFG}_0)$, and let $a$ be as in \eqref{eq:opt_feedback}.  Then, a direct calculation shows that the pair $(a,p)$ solves (\ref{eq:Kol},\ref{eq:HJ_control}).
\end{proof}

\subsection{Discounted MFG System}

Given a discount factor $\beta>0$, a flow of probabilities $p(t)$, $t\ge0$, and $x \in \cX$, the optimal control problem of the representative oscillator is given by
\begin{equation}
\label{eq:value}
    v(t,x) := \inf_{\alpha\in\mathcal{A}} \ \E \int_t^{\infty} e^{\beta (t-u)} \left(\frac{1}{2}\alpha(u,X^\alpha_u)^2 + \kappa \ell(X^{\alpha}_u,p(u))\right) \d u,
\end{equation}
where $X^\alpha_t$ is as before and $X^\alpha_t=x$.  Let $H$ be as in \eqref{eq:Hamiltonian}.  Then, the dynamic programming equation for this problem is
\begin{equation}
\label{eq:HJ_beta}
-\partial_t v(t,x) + \beta v(t,x) = H(x,v(t,x+1)-v(t,x),p(t)),
\end{equation}
for all $x \in \cX$, $t \ge0$.  Moreover,  the optimal control $a$ is given by \eqref{eq:opt_feedback}.  Then, in view of Definition \ref{def.mfgbeta}, $p(\cdot)$ is a discounted NE if it solves \eqref{eq:Kol} with this control $a$. 

We say that a pair $(v,p)$ is a \emph{classical solution} of $(\mathrm{MFG}_\beta)$ with initial condition $p(0)\in[0,1]$ if $v$ is a classical solution of the dynamic programming equation \eqref{eq:HJ_beta} and $p$ is a classical solution of \eqref{eq:Kol} with $a$ given by \eqref{eq:opt_feedback}.
As for the ergodic cost we have the following characterization.

\begin{Prop}
\label{pr:infinite}
For $\beta>0$ and a continuous flow of probabilities $p(\cdot)$, the following are equivalent:
    \begin{enumerate}[(i)]
        \item \label{i-beta}
        $p(\cdot)$ is a discounted Nash equilibrium.
        \item \label{ii-beta}
        There is $v:[0,\infty)\times \cX \mapsto \R$ such that $(v,p)$ is a bounded classical solution  of $(\mathrm{MFG}_\beta)$.
         \item \label{iii-beta}
        There is a feedback control $a\in \cA_c$ such that the pair $(a,p)$ is a bounded solution of  {\rm{(}}\ref{eq:Kol},\ref{eq:HJ_control}{\rm{)}}.
    \end{enumerate}
\end{Prop}
\begin{proof}
We follow the proof of Proposition \ref{pr:ergodic} closely.\\
$(\ref{ii-beta}) \Rightarrow (\ref{i-beta})$. Let $(v,p)$ be a bounded classical solution of $(\mathrm{MFG}_\beta)$. We apply Dynkin's formula to $e^{-\beta t}v(t,x)$ and use \eqref{eq:HJ_beta} to obtain, for any $\alpha\in\cA$,
\begin{align*}
    \E[e^{-\beta T}v(T,X^\alpha_T)]-\E[v(0,X^\alpha_0)] 
    &= \E \int_0^T e^{-\beta t} (-\beta v+\partial_t v+\ig^\alpha v)(t,X^\alpha_t)\,\d t \\
    &\geq - \E \int_0^T e^{-\beta t} \left(\frac{1}{2}\alpha(t,X^\alpha_t)^2 
    + \kappa \ell(X^{\alpha}_t,p(t))\right) \d t,
\end{align*}
where $\ig^\alpha$ as before denotes the generator of the 
controlled Markov chain $X^\alpha$, and $X^\alpha_0\sim p(0)$. Using the boundedness of $v$ we let $T\uparrow\infty$ to arrive at
$J_\beta(\alpha)\geq \E[v(0,X_0)]$,
for all $\alpha\in\cA$ with equality if $\alpha=\alpha^a\in\cA_s$ where $a$ is defined by \eqref{eq:opt_feedback}. Therefore $\alpha^a$ is optimal, and $\P(X^a_t=1)$ solves the same initial value problem as $p(t)$. Hence $p(t) = \P(X^a_t=1)$ for all $t\geq0$, proving that $p(\cdot)$ is a discounted NE.
\vspace{4pt}

\noindent
$(\ref{i-beta}) \Rightarrow (\ref{ii-beta})$. Let $p(t)=\P(X^*_t=1)$ be a discounted NE with optimal control $\alpha^*$ and optimally controlled state $X^*=X^{\alpha^*}$. It is classical that there exists a smooth bounded solution $v$ to the dynamic programming equation \eqref{eq:HJ_beta}, and it again suffices to show $\alpha^*(t,0)=a(t)^+$ and $\alpha^*(t,1)=a(t)^-$ for Lebesgue a.e.~$t\geq0$ where $a$ is defined by \eqref{eq:opt_feedback}. By the same arguments as above, we infer that $\alpha^a$ is an optimal control with $\E[v(0,X^*_0)]= J_\beta(\alpha^a)$. Towards a contraction, we assume that there exists $\delta>0$, $x'\in\cX$, and $I\subset[0,\infty)$ of positive Lebesgue measure such that \eqref{eq:contra} holds. By the same computation as in the proof of Proposition \ref{pr:ergodic}, we arrive at,
\begin{equation*}
    J_\beta(\alpha^a) = \E[v(0,X^*_0)] \leq J_\beta(\alpha^*) - \delta \int_I\P(X^*_t=x')\,\d t,
\end{equation*}
with $\delta\int_I\P(X^*_t=x')\,\d t>0$. This contradicts 
the optimality of $\alpha^*$ and shows the claim.
\vspace{4pt}

\noindent
 $(\ref{iii-beta}) \Rightarrow (\ref{ii-beta})$. Let $(a,p)$ be a bounded solution of (\ref{eq:Kol},\ref{eq:HJ_control}).  We set
 $$
    v(t,0):= \int_t^\infty e^{\beta(t-u)} H(0,-a(s),p(s))\ \d s,\qquad
    v(t,1):= \int_t^\infty e^{\beta(t-u)} H(1,a(s),p(s))\ \d s,
$$
and $A(t):= v(t,0)-v(t,1)$. We directly show that 
$$
\Dot A(t)= \Dot a(t) + \beta (A(t)-a(t)),\quad \Rightarrow\quad
\frac{\d}{\d t}( e^{- \beta t} (A(t) -a(t)))=0,\quad \Rightarrow\quad
e^{- \beta t} (A(t) -a(t))) =A(0)-a(0).
$$
Since $A$ and $a$ are bounded, we conclude that $a(t)=A(t)=v(t,0)-v(t,1)$.  Therefore,
\begin{align*}
-\partial_t v(t,0) + \beta v(t,0) = H(0,-a(t),p(t))= H(0,v(t,1)-v(t,0),p(t)),\\
-\partial_t v(t,1) + \beta v(t,1) = H(1,a(t),p(t))= H(1,v(t,0)-v(t,1),p(t)).
\end{align*} 
Hence, the pair $(v,p)$ is a classical solution of $(\mathrm{MFG}_\beta)$.
\vspace{4pt}

\noindent
$(\ref{ii-beta}) \Rightarrow (\ref{iii-beta})$.  Let $(v,p)$ be a bounded, classical solution of $(\mathrm{MFG}_\beta)$, 
and let $a$ be as in \eqref{eq:opt_feedback}.  Then, a direct calculation shows that the pair $(a,p)$ solves (\ref{eq:Kol},\ref{eq:HJ_control}).
\end{proof}

\subsection{Change of Variables}
\label{ss:cov}

We have shown that NE are characterized by bounded classical solutions to the coupled system of ordinary differential equations \eqref{eq:Kol} and \eqref{eq:HJ_control}. For further analysis, it is convenient to introduce the following change of variables to achieve symmetry,
\begin{equation*}
 q(t) = 2p(t)-1.
\end{equation*}
Then, (\ref{eq:Kol},\ref{eq:HJ_control}) is equivalent to
\begin{align} 
\nonumber\tag{\ref{eq:HJ_control}}
    \Dot a(t) &= (\beta + 2\sigma^2) a(t) + \frac{1}{2}\mathrm{sign}(a(t))a(t)^2 - \kappa q(t) ,\\
    \Dot q(t) &= a(t) - (2\sigma^2 + |a(t)|)q(t). \label{eq:HJ_controlI}
\end{align}
We study this differential equation on the strip 
$
(a,q)\in D:=\R\times[-1,+1].
$
Notice that $D$ is invariant under the above equations. With an abuse of terminology, we call a flow $q(\cdot)$ a NE if the corresponding probability
flow $p(\cdot)=(q(\cdot)+1)/2$ is a NE.

\section{Stationary Equilibria} \label{sec:Stationary_Equilibria}

A \textit{stationary mean-field game Nash equilibrium} (SNE) is a constant mean-field game equilibrium. We emphasize that in this case the initial condition $q(0)$ is not given anymore, but becomes part of the solution. In view of the results of Section \ref{sec:characterization} and the above change of variables, SNE are given by the second component of the fixed points of the planar system of ordinary differential equations (\ref{eq:HJ_control}, \ref{eq:HJ_controlI}), both in the ergodic and discounted cases.  To compute them, let $S_a$ and $S_q$ denote the nullclines of this system:
$$
    S_a = \{(a,q) \,|\, (\beta + 2\sigma^2) a + \mathrm{sign}(a)a^2/2 - \kappa q = 0 \},\qquad
    S_q := \{(a,q) \,|\, a - (2\sigma^2 + |a|)q = 0 \},
$$
and set
\begin{equation*}
    \kappa_c := 2\beta\sigma^2+4\sigma^4.
\end{equation*}
Clearly, $S_a\cap S_q$ are the fixed points of (\ref{eq:HJ_control}, \ref{eq:HJ_controlI}), and the origin is always in this set.  A direct calculation shows that $S_a\cap S_q$ is a singleton for $\kappa\le \kappa_c$.  However, if $\kappa>\kappa_c$, there are three stationary points $S_a\cap S_q= \{(-\overline{a},-\overline{q}),(0,0),(\overline{a},\overline{q})\}$, where
\begin{equation}\label{eq:abar}
        \Bar{q} := \frac{\Bar{a}}{\Bar{a}+2\sigma^2},\quad 
        \Bar{a} := -(\beta+3 \sigma^2)+\sqrt{(\beta+3\sigma^2)^2+2(\kappa-\kappa_c)}>0.
\end{equation}
Hence, we have the following immediate corollary.
\begin{Lem}
    Both in the ergodic and discounted models, the uniform distribution $q=0$ is an SNE for any $\kappa>0$.  Moreover,
\begin{enumerate}[(i)]
    \item Subcritical case: For $\kappa<\kappa_c$, the uniform distribution is the only SNE.

    \item Supercritical case: For $\kappa>\kappa_c$, there are three SNE given by $q=-\Bar{q},\,0,\,\Bar{q}$.
    \item Full synchronization: As $\kappa\uparrow\infty$, $\Bar{q}\uparrow1$.
\end{enumerate}
\vspace{3pt}

\noindent
In the above, $\beta=0$ corresponds to the ergodic cost.
\end{Lem}

\section{Analysis of the Discounted Problem}
\label{sec:discounted}

We fix $\beta>0$ and study all time-inhomogeneous discounted NE starting from an arbitrary initial distribution. In view of Proposition \ref{pr:infinite}, and the change of variables introduced in Section \ref{ss:cov}, these NE are given by the bounded solutions of  the nonlinear dynamical system  (\ref{eq:HJ_control}, \ref{eq:HJ_controlI}) with $\beta>0$.

\subsection{Linear Stability Analysis of Equilibria}
\label{ss:linear}
The local behavior of the dynamical system around the origin is described by the spectral properties of the linearized system,
\begin{equation*}
    \frac{\d}{\d t}
    \begin{pmatrix}
        a(t)\\
        q(t)
    \end{pmatrix}
    =
    \begin{pmatrix}
        \beta+2\sigma^2 & -\kappa\\
        1 & -2\sigma^2
    \end{pmatrix}
    \begin{pmatrix}
        a(t)\\
        q(t)
    \end{pmatrix}.
\end{equation*}
An analysis of the eigenvalues of this system shows that:
\begin{itemize}
    \item \textit{Subcritical regime $\kappa<\kappa_c$}: The origin is the only stationary point, and it is a saddle point (one positive, one negative eigenvalue).

    \item \textit{Supercritical regime (A) $\kappa_c<\kappa<\kappa_c+\beta^2/4$}: There are three stationary points. The origin is unstable (two positive eigenvalues). The other two stationary points are saddles.

    \item \textit{Supercritical regime (B) $\kappa>\kappa_c+\beta^2/4$}: There are three stationary points. The origin is a spiral source (two complex eigenvalues with positive real part). The other two stationary points remain saddles.
\end{itemize}

\subsection{Global Analysis}
\label{ss:global} 

We rewrite the equations (\ref{eq:HJ_control}, \ref{eq:HJ_controlI}) as
$(\Dot{a}, \Dot{q})^{\intercal}=f(a,q)$, where
$$
f(a,q):=
\begin{pmatrix}
(\beta + 2\sigma^2) a + \mathrm{sign}(a)a^2/2 - \kappa q \\
 a - (2\sigma^2 + |a|)q  
\end{pmatrix}.
$$
A direct calculation shows that $\mathrm{div}(f)\equiv \beta$. Moreover, for a given initial condition $(a,q)\in D$, we let $\Phi(t,a,q)$ be the unique solution at time $t\in I(a,q)$. Here $I(a,q)\subset\R$ is the maximal interval where the solution is defined. An \emph{orbit} of the dynamical system (\ref{eq:HJ_control}, \ref{eq:HJ_controlI}) is any set given by $\{\Phi(t,a,q) : t \in I(a,q) \}$ for some $(a,q)\in D$.

We start with a result that is repeatedly used in our arguments and which follows from the fact that $\mathrm{div}(f)>0$. 

\begin{Lem}\label{lem:no_closed_orbits} 
If $U\subset D$ is an open bounded set whose boundary is the closure of finitely many orbits, then $U=\emptyset$. In particular, the dynamical system (\ref{eq:HJ_control}, \ref{eq:HJ_controlI}) does not have closed orbits and all bounded solutions of this dynamical system converge to one of the stationary points of the system. Similarly, as time goes to minus infinity, all bounded solutions either converge to one of the stationary points or hit the lines $\{q=\pm1\}$.  
\end{Lem}
\begin{proof}
    Toward a contraposition, suppose that $U\neq\emptyset$. Then, using the Gauss' lemma (divergence theorem), we integrate along $\partial U$ to arrive at
    $$
    0< \beta\,\mathrm{Leb}(U) = \int_U \mathrm{div} f(x)\,\d x
    = \int_{\partial U} f(x) \cdot \nu(x) \,\d S(x) = 0,
    $$
    where $S$ is the ``surface'' measure and $\nu(x)$ denotes the exterior unit normal, which is orthogonal to $f(x)$ as $\partial U$ is an orbit, up to finitely many points. This contradiction implies that there are no bounded closed orbits. Also, by the Poincar\'{e}-Bendixson theorem, \cite{hale_ordinary_2009}(Theorem II.1.3), all bounded solutions of a planar dynamical system must converge to a stationary point or a bounded closed orbit.  Hence, as time goes to plus or minus infinity the bounded solutions either converge to a fixed point or hit $\{q=\pm1\}$.  Moreover, the flow cannot cross $\{q=\pm1\}$ going forward.
\end{proof}

Note that this lemma rules out homoclinic orbits, and using a symmetry argument, heteroclinic orbits connecting non-trivial fixed points as well. We now study the structure of the Nash equilibria.

\begin{Prop}[Subcritical regime]
\label{pro:sub}
Suppose that $\kappa<\kappa_c$. Then, there exists a strictly increasing function $\mathfrak{n}: [-1,1] \mapsto \R$ with $\mathfrak{n}(0)=0$ such that its graph $\{(\mathfrak{n}(q),q) : q \in [-1,1]\}$ is the stable manifold of the origin of the dynamical system (\ref{eq:HJ_control}, \ref{eq:HJ_controlI}). All Nash equilibria are included in the graph of $\mathfrak{n}$. In particular, for any $q\in[-1,1]$ there exists exactly one discounted Nash equilibrium starting from $q$, and all Nash equilibria converge to the origin.
\end{Prop}  

\begin{wrapfigure}[20]{r}{0.5\textwidth}
    \centering
    \vspace{-1em}
    \includegraphics[width=0.8\linewidth]{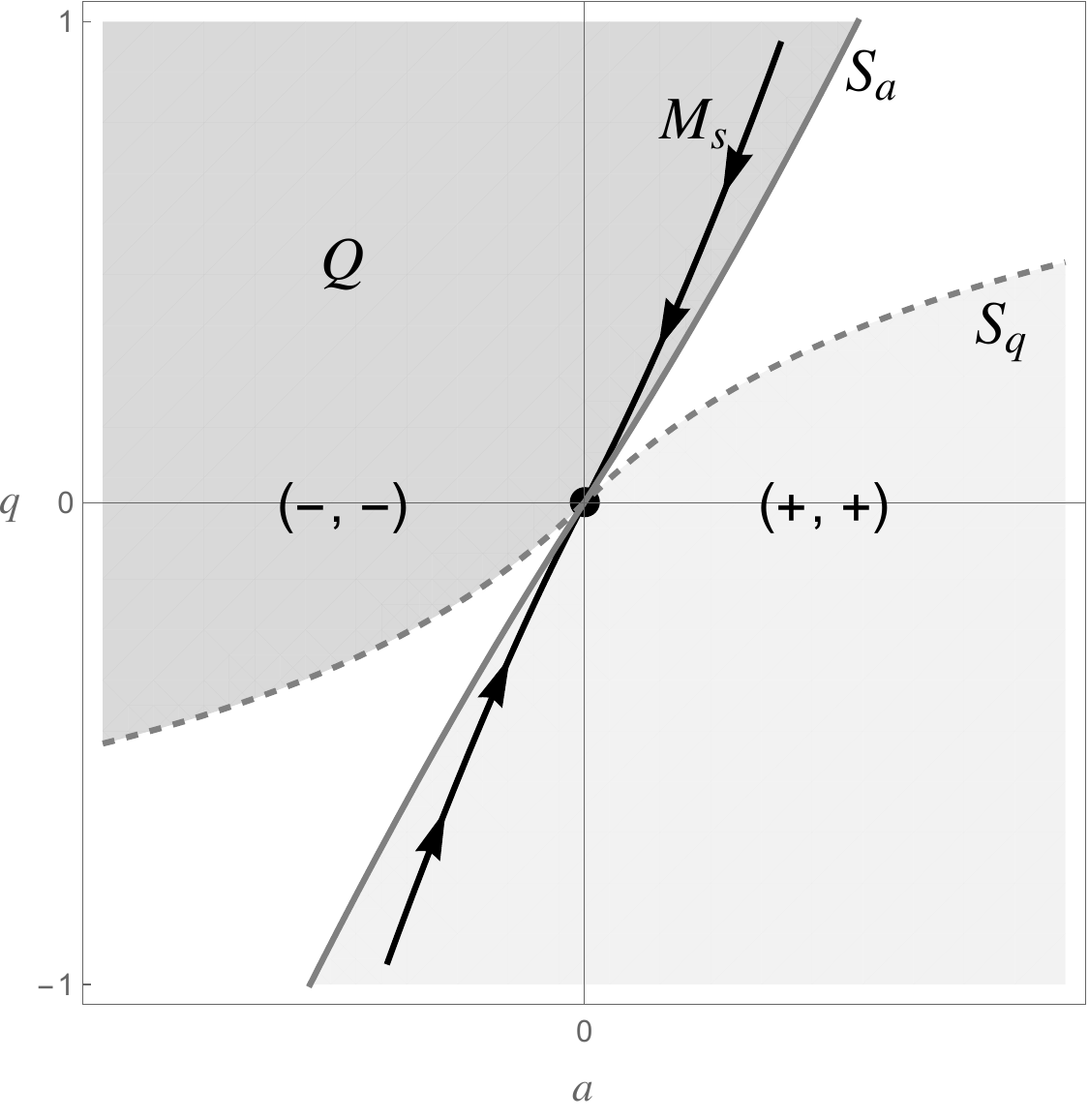}
    \caption{Subcritical discounted system (\ref{eq:HJ_control}, \ref{eq:HJ_controlI}).}
    \label{fig:Proof_Discount_Sub}
\end{wrapfigure}

\noindent\textbf{Proof.}\,\,\,
It suffices to show that the stable manifold $\mathcal{M}_s$ of the origin is a monotone curve in the phase space that hits the boundary $\{q=\pm 1\}$ as $t\to-\infty$. Additionally, by symmetry it is enough to provide the proof for the upper boundary. Let $\cQ$ denote the region above the nullclines $S_a$ and $S_q$, as shown in Figure \ref{fig:Proof_Discount_Sub}. Any trajectory passing through $\cQ$ is strictly monotone in the phase space as both components of $f$ are negative on $\cQ$, and hence is represented by a function. Moreover, linear analysis around the origin shows that the upper stable manifold $\cM_s\cap\{a>0\}$ lies locally inside $\cQ$, and therefore it is locally represented by a strictly monotone graph. If time is reversed, this manifold cannot leave $\cQ$ as $S_a$ can only be crossed vertically, and hence follows a monotone graph for all times. Since there are no fixed points of the dynamical system in $\cQ$, it cannot converge to a point on $\cQ$ and hence crosses the upper boundary $\{q=1\}$ in finite time.
The same analysis in $\{a<0\}$ completes the construction of $\mathfrak{n}$. In summary, there is a strictly increasing function $\mathfrak{n}: [-1,1] \mapsto \R$ with $\mathfrak{n}(0)=0$ such that $\cM_s= \{ (\mathfrak{n}(q),q) : q \in [-1,1]\}$. Then, for any $q\in[-1,1]$, the solution of (\ref{eq:HJ_control}, \ref{eq:HJ_controlI}) starting at $(\mathfrak{n}(q),q)$ remains in the graph of $\mathfrak{n}$ and converges to the origin as time goes to infinity. Moreover, by Lemma \ref{lem:no_closed_orbits}, all bounded solutions  must converge to the origin which is the unique stationary point. Thus they must be contained in the stable manifold $\cM_s$ of the origin. As Nash equilibria are precisely the bounded solutions of the dynamical system, we conclude that all of them must be in $\cM_s$ and thus in the graph of $\mathfrak{n}$.\hfill$\square$

\vspace{1em}

Recall that in the supercritical regime, we have three stationary points of the dynamical system: $(-\overline{a},-\overline{q})$, $(0,0)$, and $(\overline{a},\overline{q})$. We now study the supercritical regime whose relevant phase diagrams are drawn in Figures \ref{fig:Phase_Diag_disc_supA} and \ref{fig:Phase_Diag_disc_supB}.

\begin{Thm}[Supercritical regime]
\label{thm:super}
    Suppose that $\kappa>\kappa_c$. Then, there is a curve $\cC\subset D$ that connects all three stationary points of the dynamical system (\ref{eq:HJ_control}, \ref{eq:HJ_controlI}) and hits the boundary $\{q=\pm1\}$. It is given by the stable manifolds of the non-trivial equilibria joining the origin. All discounted NE are included in $\cC$ so that for any $q\in[-1,+1]$ there is at least one NE starting from $q$. Moreover,
    \begin{enumerate}[(A)]
        \item if $\kappa<\kappa_c+\beta^2/4$, then $\cC$ is a strictly monotone curve, so that there is a unique discounted NE starting from $q$, and when $q\neq0$ it converges to $\mathrm{sign}(q)\overline{q}$.

        \item if $\kappa>\kappa_c+\beta^2/4$, then $\cC$ spirals around the origin. In particular, for initial data $q$ close to the origin there exist many NE that spiral around the origin before converging to one of the self-organizing SNE. However, for $|q|$ sufficiently large, there is a unique NE starting from $q$.
    \end{enumerate}
\end{Thm}

\begin{proof}
    We first establish the picture depicted in Figure \ref{fig:Proof_Discount_Super}. Let $\mathcal{M}_u$ denote the unstable manifold of the positive stationary point $(\overline{a},\overline{q})$. We claim that $\cM_u$ extends to the left as a monotone graph that lies above $(-\overline{a},-\overline{q})$. We continue by proving this claim. Let $\cQ_1$ be the region above both $S_a$ and $S_q$, let $\cQ_2$ be the lens enclosed by $S_a$ and $S_q$ in $\{a<0\}$, let $\cQ_3$ be the area below both $S_a$ and $S_q$, and finally let $\cQ_4$ be the lens enclosed by $S_a$ and $S_q$ in $\{a> 0\}$, as shown in Figure \ref{fig:Proof_Discount_Super}. Consider a trajectory starting at a point in $\mathcal{M}_u\cap\{a<\overline{a}\}$ which is close to $(\overline{a},\overline{q})$. Linear analysis shows that locally this trajectory lies in $\cQ_1$, and it stays above $S_q\cap\{a>0\}$ which can only be intersected horizontally. We claim that once it enters $S_q\cap \{a<0\}$, it stays in $\cQ_1$ by following a monotone graph that lies above $(-\overline{a},-\overline{q})$. Indeed, while being in $\{-\overline{a}<a<0\}$, it can only leave $\cQ_1$  by entering into $\cQ_2$. Suppose this happens. Then, it either stays bounded or unbounded. Suppose it stays bounded. As the origin is unstable and the stationary point $(-\overline{a}, - \overline{q})$ can be reached from $\cQ_1$, the trajectory would have to converge to  $(\overline{a}, \overline{q})$, contradicting Lemma \ref{lem:no_closed_orbits}. Hence it has to be unbounded by entering into $\mathcal{Q}_3$ through $\mathcal{Q}_2$. 

If this happens, by symmetry, we can construct another unbounded trajectory starting from $(-\overline{a}, -\overline{q})$ which enters into $\cQ_1$ from $\cQ_4$. Then, these two trajectories would intersect, leading to a contradiction. Thus the original trajectory does not enter into $\cQ_2$ and stays in $\cQ_1$, proving the claim and establishing Figure \ref{fig:Proof_Discount_Super} where the unstable manifolds are drawn by thick lines.
      
\begin{wrapfigure}[20]{r}{0.5\textwidth}
    \centering
    \vspace{-1em}
    \includegraphics[width=.8\linewidth]{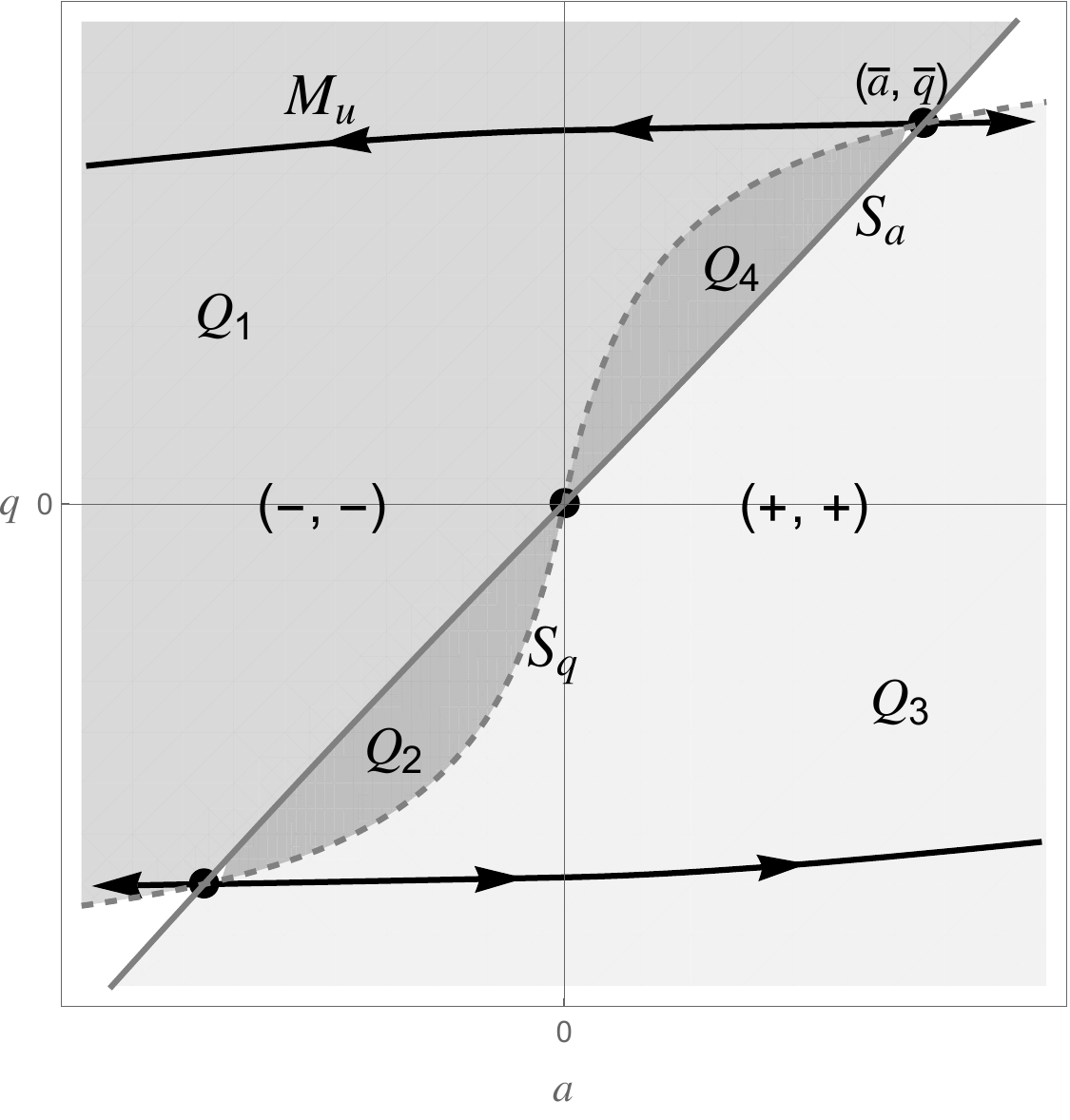}
    \caption{Supercritical discounted system (\ref{eq:HJ_control}, \ref{eq:HJ_controlI}).}
    \label{fig:Proof_Discount_Super}
\end{wrapfigure}

Next we construct $\cC$. We claim that the lower stable manifold $\cM_s\cap\{q<\overline{q}\}$ of $(\overline{a},\overline{q})$ extends to the origin. Indeed, if it is unbounded, by symmetry, we conclude that it would cross the upper stable manifold $\tilde{\cM}_s\cap\{q>-\overline{q}\}$ of $(-\overline{a},-\overline{q})$ yielding a contradiction. Hence, it must remain bounded. Moreover, it cannot intersect the unstable manifold emanating from $(-\overline{a},-\overline{q})$. This prevents it hitting the lower boundary $\{q =- 1\}$. Then, by Lemma \ref{lem:no_closed_orbits} in reverse time, it converges to any of the fixed points of the flow. We analyze all cases separately. As there are no homoclinic orbits, it cannot converge to $(\overline{a},\overline{q})$. If it converges to $(-\overline{a},-\overline{q})$ connecting two non-zero fixed points, we can construct a symmetric orbit connecting them in the opposite direction. The union of these heteroclinic orbits is also not possible by Lemma \ref{lem:no_closed_orbits}. Hence, it must converge to the origin.  

We have shown that $\cM_s$ joins the origin. By symmetry, this implies that the stable manifold $\tilde\cM_s$ of $(-\overline{a},-\overline{q})$ joins the origin as well. Moreover, a direct argument shows that the stable manifolds of the non-trivial equilibria extend to the boundary $\{q=\pm1\}$ as monotone curves. In summary, the union $\cC$ of the stable manifolds of the non-trivial fixed points together with the origin is a curve that extends from $\{q=1\}$ to $\{q=-1\}$ going through all three fixed points. Moreover, an application of Lemma \ref{lem:no_closed_orbits} shows that all bounded solutions lie on $\cC$. To construct a discounted NE starting from $q \in [-1,1]$, we choose $a^*(q)$ such that $(a^*(q),q) \in \cC$. Then, the $q$-component of the solution starting from this point is a discounted NE starting from $q$.
    
In case (A), we claim that $\cC$ constructed above is a monotone graph.  Indeed, linear analysis around the origin implies that $\cM_s$ connecting the origin to $(\overline{a},\overline{q})$ must enter into $\cQ_3$. Linear analysis around $(\overline{a},\overline{q})$ implies that $\cM_s$ enters into $\cQ_3$ in reverse time.  Hence, for these two parts to connect, all of $\cM_s\cap\{0<q<\bar q\}$ must lie in $\cQ_3$. Since the components of the vector field of the dynamical system are positive in $\cQ_3$, $\cC\cap\cQ_3$ is a monotone graph. By symmetry we conclude that $\cC$ is a monotone graph.  Further, the monotonicity of $\cC$ implies that $a^*(q)$ is unique. Since all NE starting from $q$ are given as the $q$-component of a bounded solution that lies in $\cC$, we conclude that there is a unique one for every $q$. 
    
In case (B), by the Hartman-Grobman theorem, \cite[Chapter 2.8]{perko_differential_1996}, the nonlinear dynamical system is topologically conjugate to the linearized system in a neighborhood of the origin. As the linear system spirals around the origin, so does $\cC$ of the nonlinear system. Therefore, for small $q$, there are many points $a$ so that $(a,q)\in \cC$, and for each one there is a discounted NE starting from $q$. When $|q|$ sufficiently large, $a^*(q)$ is unique and so is the NE starting from $q$. 
\end{proof}

\subsection{Analysis as a Potential Game}
\label{ss:potential}
When the running cost of a MFG is the linear derivative of a function,
the game is called potential, and
the system of equations characterizing these games
are the first-order conditions of an associated mean-field optimal control (MFC)
problem \cite{briani2018stable,FS,lasry_mean_2007}.
Indeed, all Kuramoto mean-field games discussed in this paper are potential, and
in the special case of finite-state MFGs, this is a direct consequence of the necessary part 
of Pontryagin's maximum principle, see for example \cite{gomes_socio-economic_2014,gomes_continuous_2013}  
for accounts of this fact. 

We first derive the running cost of the associated MFC using the theory developed in \cite{FS}.  
For a given $p \in [0,1]$ introduce a probability measure $\mu$ on $\cX=\{0,1\}$ by
$\mu(0)=(1-p)$ and $\mu(1)=p$.
Then, for $x \in \cX$,
$$
\ell(x,p)= p \chi_{\{x=0\}} +(1-p)\chi_{\{x=1\}} = \int_\cX|x-y| \ \mu(\d y)
= \partial_\mu \vartheta(\mu)(x),
$$
where $\partial_\mu\vartheta(\mu)$ is the linear derivative of $\vartheta$ 
as in \cite{FS}, and 
$$
\vartheta(\mu) =\frac12 \int_\cX \int_\cX
|x-y| \, \mu(\d y) \, \mu(\d x)
= p(1-p) = \frac14 (1-q^2).
$$
Hence, the two-state model that we propose
is in fact potential as we have already claimed.

We continue by defining the MFC.
As we have seen in the MFG setting, 
it is no loss of generality to restrict the feedback controls $\alpha(t,x)$
to be given as $\alpha(t,0)=\alpha^+(t)$ and $\alpha(t,1)=\alpha^-(t)$
for some measurable function $\alpha:[0,\infty) \mapsto \R$.
Let $X^\alpha_t$ be the process corresponding to this feedback control, and set 
$p^\alpha(t):= \P(X^\alpha_t=1)$, $q^\alpha(t):=2p^\alpha(t)-1$ as before. 
Then, in  terms of  $(\alpha, q^\alpha)$,
the running cost is  given by,
\begin{align*}
\frac12 \ \E[\alpha(t,X^\alpha_t)^2] + \frac{\kappa}{4} (1-(q^\alpha(t))^2) 
&= \frac14 \ [\alpha(t)^2 - \mathrm{sign}(\alpha(t))\alpha(t)^2q(t) + \kappa  (1-(q^\alpha(t))^2)].
\end{align*}
In view of \cite{FS}, the MFC problem is 
\begin{equation}
\tag{MFC$_\beta$}
\label{eq:MFC_disc}
\inf_{\alpha \in \cA}\quad  \int_0^\infty e^{-\beta t} \left(\alpha(t)^2 - 
\mathrm{sign}(\alpha(t))\alpha(t)^2q^\alpha(t) - \kappa q^\alpha(t)^2\right)\, \d t,
\end{equation}
where the controlled state process $q^\alpha(\cdot)$ is 
the solution of \eqref{eq:HJ_controlI} with $a$ replaced by $\alpha$ 
and initial condition $q(0)=q$. In view of \cite{briani2018stable,FS}
if $\alpha^*$ is an optimal control, 
the optimally controlled path $q^*(\cdot)$ is a discounted NE starting from $q$. 
Then, the minimizers of the above problem are one of the discounted NE characterized in the previous section.  
We leverage this connection to obtain several results about the minimizers. 
The following is an immediate corollary of Proposition \ref{pro:sub} and Theorem \ref{thm:super}.

\begin{Cor}
\label{cor:unique_min} Suppose that $\kappa < \kappa_c + \beta^2/4$. Then, for every $q \in [-1,1]$
there is a unique minimizer of $(\mathrm{MFC}_{\beta})$, and the optimally controlled state converges to the uniform distribution.
\end{Cor}

\begin{Prop}
    For sufficiently large $\kappa$, 
    the uniform distribution is not a minimizer and starting from any initial condition  $q \in [-1,1]$,
    the optimally controlled state $q^*(\cdot)$ converges to one of the self-organizing SNE.
\end{Prop}

\begin{proof}
We only have to consider the case of a uniform initial condition, $q(0)=0$. Then, we simply take a 
constant control $\alpha\equiv a>0$ and compute the discounted cost explicitly. 
A straightforward computation shows that we can achieve a strictly negative cost if 
\begin{equation*}
    \kappa>\kappa_c+\frac{\beta^2}{2} +\beta\sigma^2=:\widetilde{\kappa}
\end{equation*}
and $a<(\kappa-\widetilde{\kappa})/(\beta+2\sigma^2)$. This proves that the zero control cannot be optimal 
for sufficiently large $\kappa$, and the minimizer must follow one of the Nash equilibria that converge to the self-organizing SNE.
\end{proof}

\section{Analysis of the Ergodic Problem}
\label{sec:ergodic}

The critical interaction parameter is $\kappa_c=4\sigma^4$, and the dynamical  system (\ref{eq:HJ_control}, \ref{eq:HJ_controlI}) reduces to
\begin{align} 
    \Dot a(t) &= 2\sigma^2 a(t) + \frac{1}{2}\mathrm{sign}(a(t))a(t)^2 - \kappa q(t),
    \label{eq:HJ_control-erg}\\
    \Dot q(t) &= a(t) - (2\sigma^2 + |a(t)|)q(t).
    \label{eq:HJ_controlI-erg}
\end{align}
This system is conservative, and admits the following first integral which remains constant along any solution to (\ref{eq:HJ_control-erg}, \ref{eq:HJ_controlI-erg}),
\begin{equation*}
    E(a,q) = \kappa\frac{q^2}{2} + \frac{a^2}{2} - 2\sigma^2aq - \frac{a^2}{2} \mathrm{sign}(a)q.
\end{equation*}

\begin{Prop}[Ergodic subcritical regime] \label{thm:erg_sub}
    If $\kappa<\kappa_c$, then the uniform SNE is the unique ergodic Nash equilibrium. 
\end{Prop}

\begin{proof}
We already know that the origin is an ergodic NE, and the only stationary one. By Proposition \ref{pr:ergodic}, any other ergodic NE is the $q$-component of a periodic trajectory or equivalently a closed orbit. Moreover, linear stability analysis shows that the origin is a saddle point, and any solution $(a,q)$ that goes through the origin satisfies $E(a,q)=E(0,0)=0$.  Hence,  the stable manifold of the origin is given as the graph of the following function,
\begin{equation*}
    a\mapsto\frac{1}{2 \kappa } \left(\mathrm{sign}(a) a^2 +4 \sigma ^2 a + a \sqrt{a^2+8 \sigma ^2 | a| + 4(\kappa_c-\kappa)} \right).
\end{equation*}
This is a strictly increasing function in phase space  that crosses the origin and hits the boundary $\{q=\pm1\}$. Note that the stable manifold of the origin in both the subcritical ergodic system and the subcritical discounted system is represented as a monotone curve in the phase space, see Figure \ref{fig:Proof_Discount_Sub}.

If there is a closed orbit, it must enclose the origin since it is the only stationary point of this system. Then, this closed orbit  would have to cross the stable manifold, creating a contradiction.
\end{proof}
\vspace{5pt}

We next analyze the supercritical regime. In this case, a direct calculation implies that the origin is a strict local minimum of the energy $E$. Since the origin is an isolated stationary point of (\ref{eq:HJ_control-erg}, \ref{eq:HJ_controlI-erg}), all trajectories sufficiently close to the origin are closed. Then, the origin becomes a nonlinear center, and we have the following rather unexpected result.

\begin{Prop}[Ergodic supercritical regime]
\label{thm:erg_sup}
If  $\kappa>\kappa_c$, then there exist infinitely many time-periodic NE rotating around the uniform.
\end{Prop}

\subsection{Analysis as a Potential Game}
\label{ss:ergodic_potential}

This section studies Nash equilibria that arise as the first-order conditions of the following ergodic control problem, 
\begin{equation}
    \label{eq:MFC_erg}
    \tag{MFC$_0$}
    \inf_{\alpha \in \cA}\quad \underset{T\uparrow\infty}{\overline{\lim}}\ \frac1T \ \int_0^T \left(\alpha(t)^2 - \mathrm{sign}(\alpha(t))\alpha(t)^2q^\alpha(t)- \kappa q^\alpha(t)^2\right)\,\d t,
\end{equation}
where $q^\alpha(\cdot)$ is the solution of \eqref{eq:HJ_controlI} with $a=\alpha$ and any initial data. Notice that since $\sigma>0$, changes in the initial condition do not affect the ergodic cost. Moreover, there are minimizers of the above control problem. As explained in Section \ref{ss:potential}, any optimally controlled path $q^*(\cdot)$ of \eqref{eq:MFC_erg} that is periodic is an ergodic NE. The next proposition shows that the control $\alpha\equiv0$ is optimal in the subcritical regime while it becomes strictly suboptimal in the supercritical regime. This implies that either the periodic NE or one of the non-uniform stationary NE are chosen as the minimizers.
\vspace{5pt}

\begin{Thm}[Ergodic mean-field control problem] 
Fix any initial condition. If $\kappa<\kappa_c$, the zero control is optimal and the associated state converges to the uniform distribution. In the supercritical regime $\kappa>\kappa_c$, however, the zero control is not optimal and the state does not converge to the uniform distribution.
\end{Thm}

\begin{proof}   
If $\kappa<\kappa_c$, we have shown in the proof of Proposition \ref{thm:erg_sub} that any bounded solution $(a,q)$ to (\ref{eq:HJ_control-erg}, \ref{eq:HJ_controlI-erg}) converges to the origin. Hence the cost of the ergodic control problem \eqref{eq:MFC_erg} is zero, and we can achieve this cost by taking the zero control. 

Now let $\kappa>\kappa_c$. Let $\overline{a}>0$ denote the stationary control corresponding to the positive supercritical SNE, given by \eqref{eq:abar}. We take the constant control $\alpha\equiv \overline{a}>0$ and directly compute that 
\begin{equation*}
    \overline{a}^2 - \overline{a}^2q- \kappa \overline{q}^2<0 \quad \Longleftrightarrow \quad \kappa>\kappa_c.
\end{equation*} 
This implies that we can achieve a strictly negative ergodic cost if $\kappa>\kappa_c$, proving the claim.
\end{proof}

\subsection{Convergence of the Finite Horizon Game}
\label{ss.finite}

\begin{wrapfigure}[20]{r}{0.49\textwidth}
    \centering
    \includegraphics[width=.7\linewidth]{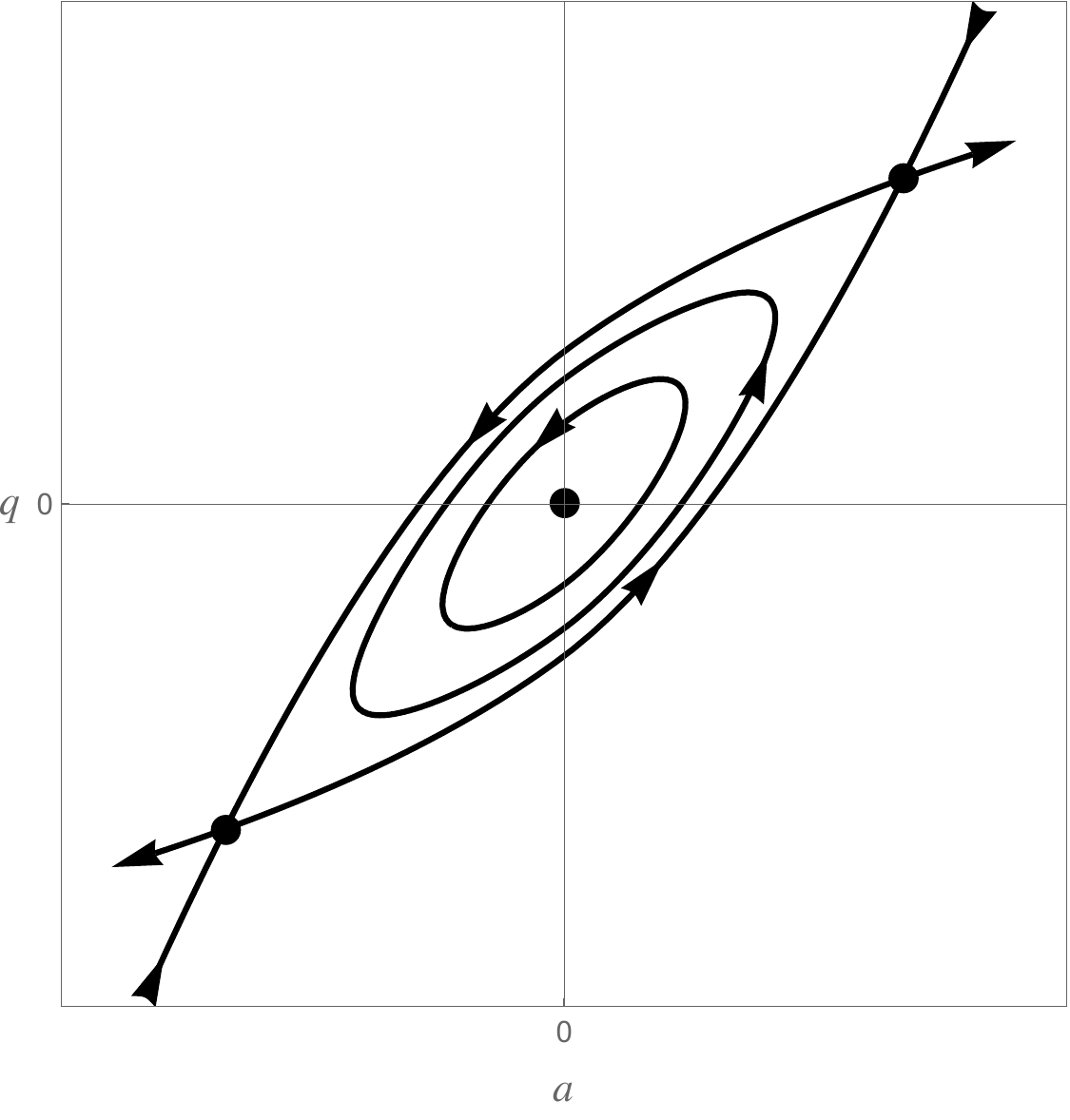}
    \caption{Stable and unstable manifolds of the non-trivial equilibria of (\ref{eq:HJ_control-erg}, \ref{eq:HJ_controlI-erg}), and periodic orbits in the supercritical regime $\kappa >\kappa_c$.}
    \label{fig:erg_turnpike}
\end{wrapfigure}

Fix $T>0$ and for a given flow of probabilities $(p(t))_{t\in[0,T]}$ consider the finite horizon problem of the representative agent:
\begin{equation*}
    v^T(t,x) := \inf_{\alpha\in\mathcal{A}}\  \E \int_t^T\left(\frac12 \alpha(u,X^\alpha_u)^2 + \kappa \ell(X^{\alpha}_u,p(u))\right) \d u,
\end{equation*}
where $X^\alpha_t=x$. Then, the finite horizon Nash equilibria are defined analogously, and following the proof of Proposition \ref{pr:ergodic}, we conclude that finite horizon Nash equilibria are in one-to-one correspondence with bounded trajectories $(a^T,q^T)$ of the conservative system (\ref{eq:HJ_control-erg}, \ref{eq:HJ_controlI-erg}), together with an initial condition for $q^T(0)$ and the terminal condition $a^T(T)=0$. 

In the supercritical case, a direct analysis of the energy $E$ reveals that the stable and unstable manifolds of the non-trivial equilibria join each other, creating a lens as depicted in Figure \ref{fig:erg_turnpike}. Inside this lens, any orbit is periodic. Then, one observes that as the periodic orbits get closer to the boundary of this lens, they spend most of their time near the two self-organizing SNE. This can be seen by comparing trajectories that start on the stable manifolds of the non-trivial SNE (i.e., on the boundary of the lens) to periodic trajectories with close initial data. The claim then follows by continuous dependence of solutions on the initial data and the fact that solutions on the stable manifold approach the stationary points exponentially slowly. Further, as periodic orbits inside the lens approach the boundary of the lens, their period goes to infinity. Since finite horizon NE are characterized by solutions $(a^T,q^T)$ to the same ODE system, but with the terminal condition $a^T(T)=0$, we therefore have the following \emph{turnpike property}: For any self-organizing SNE and any initial distribution there exists a sequence of finite horizon NE $q^T(\cdot)$, $T\in\N$, that spend an arbitrarily large amount of time at either self-organizing SNE as $T\uparrow\infty$. This property is somehow in analogy with the local stability result proved in \cite[Theorem 3.1]{cesaroni_stationary_2023} for the original Kuramoto mean-field game in the large coupling constant regime. However, for initial conditions in $(-\overline{q},\overline{q})$, there are many other sequences of finite horizon Nash equilibria that rotate and stay bounded away from any SNE as $T\uparrow\infty$ as well. 

\section{Extensions and Conclusion}
\label{sec:eg}

We continue to discuss possible extensions of our model. To incorporate asymmetric settings in which the population naturally favors a particular state, one might consider different values for the thermal noise $\sigma^2$. In this case, the controlled rates are given by,
$$
\lambda^\alpha_{0,1}(t) = \sigma_0^2 + \alpha(t,0),\quad \lambda^\alpha_{1,0}(t) = \sigma_1^2 + \alpha(t,1),
$$
for different values of $\sigma_0^2$ and $\sigma_1^2$. A straightforward modification of our analysis shows that, while NE are not symmetric around the uniform distribution anymore, these models share the same qualitative structure concerning uniqueness and stability of the NE.

A natural extension of our model are problems in which agents are allowed to choose from more than two states as discussed in the Introduction. If $N>2$ denotes the number of states, an immediate extension of Proposition \ref{pr:infinite} shows that Nash equilibria are characterized by solutions to a $2N$-dimensional coupled forward-backward ODE system, see also \cite[Section 2.3]{cohen_analysis_2023} for details. 
As in the Introduction, let $\cX_N=\{x_1,..,x_N\}$. Then, for $x_i\in\cX_N$ and $v\in\R^N$ define a Hamiltonian,
\begin{align*}
H(x_i,v,p) &:= \inf_{a_\ell,a_r\geq0} 
\left\{(N^2\sigma^2 + Na_\ell)(v_{i-1}-v_i) 
+ \frac12 a_\ell^2 
+ (N^2\sigma^2 + Na_r)(v_{i+1}-v_i) 
+ \frac12 a_r^2 \right\} + \kappa \ell( x_i, p) \\
&= \frac{N^2}{2} \left(2 \sigma^2 (v_{i-1}-2v_i+v_{i+1})
 - [(v_{i-1}-v_i)^-]^2 - [(v_i-v_{i+1})^+]^2\right) + \kappa \ell( x_i, p),
\end{align*}
where $i-1$, $i+1$ are to be understood modulo $N$, and the running cost $\ell$ is defined as in \eqref{eq:Kuramoto_cost}.

The discounted Nash equilibria are described by 
the solutions of the following forward-backward system, 
\begin{align*}
    \partial_t v(t,x_i) &= \beta v(t,x_i) - H(x_i,v(t),p(t)) \\
    \partial_t p(t,x_i) &= - \left(2 N^2\sigma^2 + N(v(t,x_{i-1})-v(t,x_i))^- + N(v(t,x_i)-v(t,x_{i+1}))^+ \right) \, p(t,x_i) \\
    &\qquad +\left(N^2\sigma^2 + N(v(t,x_{i-1})-v(t,x_i))^+ \right) \, p(t,x_{i-1})
    + \left( N^2\sigma^2 + N(v(t,x_i)-v(t,x_{i+1}))^- \right) \, p(t,x_{i+1}),
\end{align*}
for $i=1,..,N$. The unknowns are $v(t)=(v(t,x_1),..,v(t,x_N))$ and $p(t)=(p(t,x_1),..,p(t,x_N))$ 
with an initial condition for $p(0)$. 

As in the two state problem, the above system of equations
reduces to a $2(N-1)$-dimensional system.
Indeed, the functions 
$$
a(t,x_i)=v(t,x_i)-v(t,x_{i+1}),\qquad i=1,.., N,\, t\ge 0
$$
solve a closed system obtained from the dynamic programming equation
satisfied by $v$: for $i=1,.., N$, $t\ge 0$,
\begin{align*}
\partial_t a(t,x_i) =&\beta a(t,x_i) + \kappa (\ell(x_i,p)-\ell(x_{i+1},p))
+N^2\sigma^2 [a(t,x_{i-1})-2a(t,x_i)+a(t,x_{i+1})]\\
&+\frac{N^2}{2} [-(a(t,x_{i-1})^-)^2+(a(t,x_{i+1})^+)^2 -\mathrm{sign}(a(t,x_{i})) (a(t,x_{i}))^2].
\end{align*}
Additionally, by their definitions, 
$\sum_i a(t,x_i)=0$ as well as $\sum_i p(t,x_i)=1$ hold for all $t\geq0$. 
Hence, the dimension of the system of equations
satisfied by $(a,p)$ is $2(N-1)$.
Moreover, the optimal feedback control is given by
$$
\alpha_r(t,x_i)= a(t,x_i)^+, \qquad 
\alpha_\ell(t,x_i)= a(t,x_{i-1})^-, \qquad
i=1,..,N. 
$$

One directly verifies that for all $\kappa>0$,
$a\equiv 0$, $p\equiv 1/N$ is a solution 
corresponding the uniform distribution,
and $v\equiv \kappa/\beta$.
Formal calculations indicate that
when $\kappa$ is zero or small,
there are no stationary solutions near the
uniform.  This indicates that there could
be phase transition as in the two and continuous state
problems.  However,
the analysis of this system  requires different methods as the ones presented in this paper which heavily draw from the theory of planar ordinary differential equations.

In conclusion,
 this work studies a two-state mean-field game as a tractable model for synchronization obtained as the  discretization of the Kuramoto mean-field game introduced and studied in \cite{yin_synchronization_2010,carmona_synchronization_2023,cesaroni_stationary_2023}. After characterizing Nash equilibria (NE) as the bounded solutions of a planar dynamical system (\ref{eq:Kol},\ref{eq:HJ_control}), we provide a complete analysis of all cases.  In particular, the game exhibits phase transitions as the classical Kuramoto model \cite{kuramoto_self-entrainment_1975}.  While the uniform distribution, representing the incoherent state, is the unique stationary Nash equilibrium (SNE) for small values of the coupling constant, coherent SNE emerge as this interactions gets stronger. Moreover, in the supercritical regime, there are many time-inhomogeneous NE starting from the same initial distribution. Surprisingly, there are also periodic NE with the ergodic cost.
 
\appendix
{\section{Ergodic Hamilton-Jacobi Equations}
\label{app:erg}

This appendix discusses well-posedness of equation \eqref{eq:HJ_erg} 
using a vanishing discount argument. 

\begin{Lem}[Existence] \label{lem:vanish-disc}
Let $p(t)\in[0,1]$ be a continuous, periodic flow of probabilities. Then, there exists a classical solution $(\overline{\lambda},v)$ to \eqref{eq:HJ_erg} such that $t\mapsto v(t,x)$ is periodic with same period as $p(t)$.
\end{Lem}

\begin{proof}
Let $v_\beta=v_\beta(t,x)$ be the value function of infinite horizon problem
with discount factor $\beta>0$. 
By uniqueness of the dynamic programming equation
\eqref{eq:HJ_beta}, $t \mapsto v_\beta(t,x)$ is periodic with the same period as $p(t)$. 
Let $\tau$ denote the common period. 
We continue in several steps.
\vspace{4pt}

\emph{Step 1.} It is well-known that $v_\beta$ is the unique classical solution the 
dynamic programming equation \eqref{eq:HJ_beta}.
 Let $a_\beta(t):=v_\beta(t,0)-v_\beta(t,1)$ be the optimal feedback control 
 solving \eqref{eq:HJ_control}. 
 We further observe that $a_\beta$ is $\tau$-periodic. 
\vspace{4pt}

\emph{Step 2.} Set $w_\beta(t,x):=v_\beta(t,x)-v_\beta(0,0)$
and $\|\cdot\|_\infty$ be the sup-norm.  
First, by the definition of $v_\beta$,
\begin{equation}\label{eq:estimate_1}
0\leq \beta v_\beta(t,x) = \inf_{\alpha\in\mathcal{A}} \E \int_t^{\infty} \beta e^{\beta (t-u)} \left(\frac{1}{2}\alpha(u,X^\alpha_u)^2 + \kappa \ell(X^{\alpha}_u,p(u))\right) \d u \leq \int_t^\infty \beta\kappa e^{\beta (t-u)} \,\d u = \kappa.
\end{equation}
\vspace{4pt}

\emph{Step 3.}
Since $a_\beta$ is continuously differentiable and $\tau$-periodic, 
the minimal and maximal values of $a_\beta$ are achieved 
at some point $t_*$.
Let $a_*=a(t_*)$ be such a point.  Then, by \eqref{eq:HJ_control},
$$
 (\beta + 2\sigma^2) a_* + \frac{1}{2}\mathrm{sign}(a_*)a_*^2 - \kappa (2p(t)-1) =0.
 $$
 Suppose $a_*>0$.  Then, 
$a_*^2 +  2(\beta + 2\sigma^2) a_* - 2\kappa (2p(t_*)-1) =0$.
If $2p(t_*) <1$ both roots are negative.  If $2p(t_*)>1$, then
$$
0\le  a_*= - (\beta + 2\sigma^2) + \sqrt{ (\beta + 2\sigma^2)^2+ 2\kappa (2p(t_*)-1)} 
\le  2 \sqrt{ \kappa}. 
$$
Similarly, for  $a_*<0$ the equation is
$a_*^2 -  2(\beta + 2\sigma^2) a_* + 2\kappa (2p(t_*)-1) =0$.
 If $2p(t_*) >1$ both roots are positive, and if $2p(t_*)<1$, 
$$
0\ge  a_*= (\beta + 2\sigma^2) - \sqrt{ (\beta + 2\sigma^2)^2- 2\kappa (2p(t_*)-1)} 
\ge - 2 \sqrt{ \kappa}. 
$$
Therefore, $\|a_\beta\|_\infty \le 2 \sqrt{ \kappa}$. 
In view \eqref{eq:HJ_control},
$|\partial_t a_\beta| \le (\beta +2\sigma^2) \|a_\beta\|_\infty +  \|a_\beta\|_\infty^2/2 + \kappa$.
Hence, 
$$
\sup_{\beta\in(0,1)} [\|a_\beta\|_\infty  
+ \|\partial_t a_\beta\|_\infty]
< \infty.
$$
In particular, $a_\beta$ is a bounded, equicontinuous family
of functions.
\vspace{4pt}

\emph{Step 4.}
We first observe that by \eqref{eq:HJ_beta},
$$
\partial_t v_\beta(t,x)= \beta v_\beta(t,x) -H(x, v_\beta(t,x+1)-v_\beta(t,x),p(t)),
$$
Then, the estimates obtained in the previous step
 and \eqref{eq:estimate_1} imply that
$$
\sup_{\beta\in(0,1)} \|\partial_t v_\beta\|_\infty  < \infty.
$$
We use these and the  $\tau$-periodicity of $v_\beta$, to arrive at
$$
\sup_{\beta\in(0,1)} \ \|w_\beta\|_\infty\leq 
\sup_{\beta\in(0,1)} [\|a_\beta\|_\infty + \tau \,  \|\partial_t v_\beta\|_\infty ] < \infty.
$$ 
Moreover, $\partial_t w_\beta=\partial_t v _\beta$. Hence,
$$
\sup_{\beta\in(0,1)} [\|w_\beta\|_\infty  
+ \|\partial_t w_\beta\|_\infty]
< \infty,
$$
and consequently, $w_\beta$ is also a bounded, equicontinuous family
of functions.
\vspace{4pt}

\emph{Step 5.}
 Since $w_\beta$ is uniformly bounded,
 $\beta w_\beta(t,x)$ converges to $0$ as $\beta\downarrow0$,
 and therefore, any limit point of $\beta v_\beta(t,x)$ is independent of $(t,x)$. 
 Additionally, in view of the equicontiniuity of $a_\beta,w_\beta$ and the Arzel\`a–Ascoli theorem
 there is a sequence $\beta_k\downarrow0$ and continuous functions $v(t,x)$ and $a(t)$ such that 
$$
\lim_{k\to\infty} \ [\, \|w_{\beta_k}(\cdot,x)-v(\cdot,x)\|_\infty + \|a_{\beta_k}-a\|_\infty \, ] = 0,
$$
and such that
$\beta_k v_{\beta_k}(t,x)$ converges to a constant $\overline{\lambda}\in\R$. 
\vspace{4pt}

\emph{Step 6.} We claim that $v(\cdot,x)$ is continuously differentiable with derivative 
$$
\partial_t v(t,x) = \overline{\lambda}-H(x,v(t,x+1)-v(t,x),p(t)) =: F(t,x).
$$
Indeed by \eqref{eq:HJ_beta}, $w_{\beta_k}$ solves
$$
\partial_t w_{\beta_k}(t,x)  =  \beta v_{\beta_k}(t,x) -H(x,w_{\beta_k}(t,x+1)-w_{\beta_k}(t,x),p(t)) =: F_k(t,x). 
$$
Then, for any compactly supported $\phi\in C^\infty_c((0,\infty))$,
$$
    \int_0^\infty \partial_t w_{\beta_k} (t,x) \phi(t) \,\d t = - \int_0^\infty w_{\beta_k} (t,x) \phi'(t) \,\d t \ \longrightarrow \ - \int_0^\infty v (t,x) \phi'(t) \,\d t, \quad \text{as} \ k\to\infty,
$$
by the dominated convergence theorem. At the same time, 
$$
\int_0^\infty \partial_t w_{\beta_k} (t,x) \phi(t) \,\d t = \int_0^\infty F_k(t,x) \phi(t) \,\d t \ \longrightarrow \ \int_0^\infty F(t,x) \phi(t) \,\d t, \quad \text{as} \ k\to\infty,
$$
by the dominated convergence theorem. Hence, $F(\cdot,x)$ is the weak (in the sense of distributions) derivative of $v(\cdot,x)$ on $(0,\infty)$. But $F(\cdot,x)$ is continuous, hence the strong derivative of $\partial_t v(t,x)$ exists and is equal to $F(t,x)$ on $(0,\infty)$. This extends to $[0,\infty)$ and hence implies that $(\overline{\lambda},v)$ is a classical solution to \eqref{eq:HJ_erg}.
\end{proof}

We continue to discuss the uniqueness of a solution $(\overline{\lambda},v)$ to \eqref{eq:HJ_erg}.

\begin{Lem}[Uniqueness]
    Let $p(t)$ be a continuous, periodic flow of probabilities
    and $(\lambda_u,u)$, $(\lambda_w,w)$ be two
    classical solutions of  \eqref{eq:HJ_erg}
    that are periodic with same period.
    Then,  $\lambda_v=\lambda_w$ and
    $$
    (w-u)(t,x)= (w-u)(0,0), \qquad \forall \ x \in \cX, t \ge0.
    $$
\end{Lem}

\begin{proof}
Following the verification argument presented in $(\ref{ii}) \Rightarrow (\ref{i})$
 of Proposition \ref{pr:ergodic}, we conclude that both $\lambda_u$ and 
 $\lambda_w$ are equal to the ergodic cost $\overline{\lambda}$
 constructed in the existence proof.
 \vspace{4pt}
 
 \emph{Step 1}.  Suppose that at some point $t_*$,
 $$
 (w-u)(t_*,0)=(w-u)(t_*,1)=:\gamma_*.
 $$
 We periodically extend both $w$ and $u$ to all $t \in \R$, and
 for $x\in \cX,\, t \le t_*$ set,
 $$
 \tilde u(t,x):= u(t,x)+\gamma_* +\overline{\lambda}(t_*-t),\qquad
  \tilde w(t,x):= w(t,x)+\overline{\lambda}(t_*-t).
 $$ 
 Then both $\tilde u$ and $\tilde w$ solve the same
 finite horizon dynamic programming equation,
 $$
 -\partial_t v(t,x) =H(x,v(t,x+1)-v(t,x), p(t)), \qquad x\in \cX,\, t < t_*.
 $$
 Moreover, $ \tilde u(t_*,x)= \tilde w(t_*,x)$ for $x \in \cX$.
 Then, by the uniqueness of  solutions
 of the above equation, we conclude that
 for every $x\in \cX,\, t \le t_*$ we have
 $$
 \tilde u(t,x)= \tilde w(t,x),
 \quad \Leftrightarrow \quad
 u(t,x)+\gamma_*=w(t,x).
 $$
 
 \emph{Step 2}. Let $\tau$ be the common
 period of $u$ and $w$, and we extend them to all of $\R$.  Then, there are 
 $t_* \in [0,\tau]$,  $x_*\in \cX$ such that
 $$
 \gamma_*:=  (w-u)(t_*,x_*)
 = \max\{ (w-u)(t,x)\, :\, (t,x) \in \R \times \cX\, \}.
 $$
 Clearly $\partial_t (w-u)(t_*,x_*)=0$ and by \eqref{eq:HJ_erg},
 $$
 0=-\partial_t (w-u)(t_*,x_*) =
 H(x_*,w(t_*,x_*+1)-w(t_*,x_*),p(t_*)) -H(x_*,u(t_*,x_*+1)-u(t_*,x_*),p(t_*)).
 $$
Invertibility of $a \mapsto H(x_*,a,p(t_*))$ implies that 
$w(t_*,x_*+1)-w(t_*,x_*)=u(t_*,x_*+1)-u(t_*,x_*)$. Therefore,
$$
(w-u)(t_*,x_*+1)=(w-u)(t_*,x_*)=\gamma_*.
$$
We now use Step 1 to conclude.
\end{proof}

\section{Proof of Lemma \ref{lem:equal}}
\label{app:equal}

We fix $\alpha \in \cA$.
Set 
$$
p(t):= \P(X^\alpha_t=1),\qquad
R(t):= \frac{\alpha(t,0)}{\alpha(t,1)}- \frac{p(t)}{1-p(t)},
\qquad t \ge 0.
$$
We next define $\tilde \alpha$ by,
\begin{align*}
(\tilde \alpha(t,0)\, , \, \tilde \alpha(t,1))&:=
(\alpha(t,0)-\alpha(t,1)\frac{p(t)}{1-p(t)}\, ,\, 0), \qquad\text{if}\ R(t)>0,\\
(\tilde \alpha(t,0)\, , \, \tilde \alpha(t,1))&:=
(0\, ,\, \alpha(t,1)-\alpha(t,0)\frac{1-p(t)}{p(t)}), \qquad\text{if}\ R(t)\le0.
\end{align*}
We directly verify that $\tilde \alpha(t,0),\tilde \alpha(t,1) \ge 0$ and
$\tilde \alpha \in \cA_s$.  Moreover, by the definitions,
$$
\Dot{p}(t)= (\sigma^2+\alpha(t,0)) (1-p(t)) -
(\sigma^2+\alpha(t,1)) p(t)
=  (\sigma^2+\tilde \alpha(t,0)) (1-p(t)) -
(\sigma^2+\tilde \alpha(t,1)) p(t).
$$
Therefore, $p(t)=\P(X^\alpha_t=1)=\P(X^{\tilde \alpha}_t=1)$.
Also, by the construction it is clear that
$|\tilde \alpha(t,x)|\le |\alpha(t,x)|$ for all $x \in \cX$ and $t \ge 0$.

Suppose that there is $\alpha_p \in \cA_s$ so that
the pair $(p, \alpha_p)$ satisfies \eqref{eq:KOL}.
Then, by \eqref{eq:KOL},
$$
(\Dot{p}(t)+\sigma^2(2p(t)-1)) = \alpha_p(t,0)(1-p(t)) - \alpha_p(t,1) p(t).
$$
Since $\alpha_p(t,0) \, \alpha_p(t,1)=0$, this implies that
$$
\alpha_p(t,0):= 
        \frac{(\Dot{p}(t) +\sigma^2(2p(t)-1))^+}{1-p(t)},
\qquad
\alpha_p(t,1):= \frac{(\Dot{p}(t) +\sigma^2(2p(t)-1))^-}{p(t)}.
$$
Hence $\alpha_p$ is uniquely determined by $p(t)$.
\hfill $\Box$}

\printbibliography

@article{BC,
  title={Mean {F}ield {G}ames {M}aster {E}quations: {F}rom {D}iscrete to {C}ontinuous {S}tate {S}pace},
  author={Bertucci, Charles and Cecchin, Alekos},
  journal={arXiv preprint arXiv:2207.03191},
  year={2022}
}

@article{FS,
  title={Potential {M}ean-{F}ield {G}ames and {G}radient {F}lows},
  author={H\"{o}fer, Felix and Soner, H Mete},
  journal={arXiv:2408.00733},
  year={2024}
}

@article{castellano_statistical_2009,
	title = {Statistical {P}hysics of {S}ocial {D}ynamics},
	volume = {81},
	number = {2},
	journal = {Reviews of Modern Physics},
	author = {Castellano, Claudio and Fortunato, Santo and Loreto, Vittorio},
	year = {2009},
	pages = {591--646}
}

@article{bauso_opinion_2016,
	title = {Opinion {Dynamics} in {Social} {Networks} through {Mean}-{Field} {Games}},
	volume = {54},
	number = {6},
	journal = {SIAM Journal on Control and Optimization},
	author = {Bauso, D. and Tembine, H. and Ba\c{s}ar, T.},
	year = {2016},
	pages = {3225--3257}
}

@article{acemoglu_opinion_2013,
	title = {Opinion {Fluctuations} and {Disagreement} in {Social} {Networks}},
	volume = {38},
	number = {1},
	journal = {Mathematics of Operations Research},
	author = {Acemo\u{g}lu, Daron and Como, Giacomo and Fagnani, Fabio and \"{O}zda\u{g}lar, Asuman},
	year = {2013},
	pages = {1--27}
}

@book{hale_ordinary_2009,
	title = {Ordinary {Differential} {Equations}},
	publisher = {Dover Publications},
	author = {Hale, Jack K.},
	year = {2009}
}

@book{perko_differential_1996,
	title = {Differential {Equations} and {Dynamical} {Systems}},
	volume = {7},
	publisher = {Springer US},
	author = {Perko, Lawrence},
	year = {1996}
}

@inproceedings{huang_nash_2007,
	title = {The {Nash} {C}ertainty {E}quivalence {P}rinciple and {McKean}-{Vlasov} {S}ystems: {An} {I}nvariance {P}rinciple and {E}ntry {A}daptation},
	booktitle = {2007 46th {IEEE} {Conference} on {Decision} and {Control}},
	author = {Huang, Minyi and Caines, Peter E. and Malham\'{e}, Roland P.},
	year = {2007},
	pages = {121--126}
}

@article{huang_large-population_2007,
	title = {Large-{Population} {Cost}-{Coupled} {LQG} {Problems} {With} {Nonuniform} {Agents}: {Individual}-{Mass} {Behavior} and {Decentralized} $\varepsilon$-{Nash} {Equilibria}},
	volume = {52},
	issn = {1558-2523},
	journal = {IEEE Transactions on Automatic Control},
	author = {Huang, Minyi and Caines, Peter E. and Malham\'{e}, Roland P.},
	year = {2007},
	pages = {1560--1571}
}

@article{huang_invariance_2007,
	title = {An {Invariance} {Principle} in {Large} {Population} {Stochastic} {Dynamic} {Games}},
	volume = {20},
	number = {2},
	journal = {Journal of Systems Science and Complexity},
	author = {Huang, Minyi and Caines, Peter E. and Malham\'{e}, Roland P.},
	year={2007},
	pages = {162--172}
}

@inproceedings{huang_individual_2003,
	title = {Individual and {M}ass {B}ehaviour in {L}arge {P}opulation {S}tochastic {W}ireless {P}ower {C}ontrol {P}roblems: {C}entralized and {Nash} {E}quilibrium {S}olutions},
	volume = {1},
	booktitle = {42nd {IEEE} {International} {Conference} on {Decision} and {Control}},
	author = {Huang, Minyi and PE, Caines and Malham\'{e}, R.P.},
	year = {2003},
	pages = {98--103}
}

@article{lasry_mean_2007,
	title = {Mean {F}ield {G}ames},
	volume = {2},
	number = {1},
	journal = {Japanese Journal of Mathematics},
	author = {Lasry, Jean-Michel and Lions, Pierre-Louis},
	year = {2007},	
	pages = {229--260}
}

@article{lasry_jeux_2006,
	title = {Jeux \`{a} champ moyen. {II} – {Horizon} fini et contr\^{o}le optimal},
	volume = {343},
	number = {10},
	journal = {Comptes Rendus Mathematique},
	author = {Lasry, Jean-Michel and Lions, Pierre-Louis},
	year = {2006},
	pages = {679--684}
}

@article{lasry_jeux_2006-1,
	title = {Jeux \`{a} champ moyen. {I} – {Le} cas stationnaire},
	volume = {343},
	number = {9},
	journal = {Comptes Rendus Mathematique},
	author = {Lasry, Jean-Michel and Lions, Pierre-Louis},
	year = {2006},
	pages = {619--625}
}

@article{cecchin_convergence_2019,
	title = {On the {Convergence} {Problem} in {Mean} {Field} {Games}: {A} {Two} {State} {Model} without {Uniqueness}},
	volume = {57},
	number = {4},
	journal = {SIAM Journal on Control and Optimization},
	author = {Cecchin, Alekos and Pra, Paolo Dai and Fischer, Markus and Pelino, Guglielmo},
	year = {2019},
	pages = {2443--2466}
}

@article{gueant_existence_2015,
	title = {Existence and {Uniqueness} {Result} for {Mean} {Field} {Games} with {Congestion} {Effect} on {Graphs}},
	volume = {72},
	number = {2},
	journal = {Applied Mathematics \& Optimization},
	author = {Gu\'{e}ant, Olivier},
	year = {2015},
	pages = {291--303}
}

@article{cecchin_probabilistic_2020,
	title = {Probabilistic {Approach} to {Finite} {State} {Mean} {Field} {Games}},
	volume = {81},
	issn = {1432-0606},
	number = {2},
	journal = {Applied Mathematics \& Optimization},
	author = {Cecchin, Alekos and Fischer, Markus},
	year = {2020},
	pages = {253--300}
}

@article{gomes_socio-economic_2014,
	title = {Socio-{E}conomic {A}pplications of {F}inite {S}tate {M}ean {F}ield {G}ames},
	volume = {372},
	number = {2028},
	journal = {Philosophical Transactions of the Royal Society A: Mathematical, Physical and Engineering Sciences},
	author = {Gomes, Diogo and Velho, Roberto M. and Wolfram, Marie-Therese},
	year = {2014},
	pages = {20130405}
}

@article{dai_pra_climb_2019,
	title = {Climb on the {Bandwagon}: {Consensus} and {Periodicity} in a {Lifetime} {Utility} {Model} with {Strategic} {Interactions}},
	volume = {9},
	number = {4},
	journal = {Dynamic Games and Applications},
	author = {Dai Pra, Paolo and Sartori, Elena and Tolotti, Marco},
	year = {2019},
	pages = {1061--1075}
}

@book{carmona_probabilistic_2018-1,
	title = {Probabilistic {Theory} of {Mean} {Field} {Games} with {Applications} {I}},
	volume = {83},
	publisher = {Springer International Publishing},
	author = {Carmona, Ren\'{e} and Delarue, Fran\c{c}ois},
	year = {2018}
}

@article{kolokoltsov_mean-field-game_2017,
	title = {Mean-{Field} {Game} {Model} of {Corruption}},
	volume = {7},
	number = {1},
	journal = {Dynamic Games and Applications},
	author = {Kolokoltsov, V. N. and Malafeyev, O. A.},
	year = {2017},
	pages = {34--47}
}

@article{kolokoltsov_corruption_2018,
	title = {Corruption and {B}otnet {D}efense: {A} {M}ean {F}ield {G}ame {A}pproach},
	volume = {47},
	number = {3},
	journal = {International Journal of Game Theory},
	author = {Kolokoltsov, V. N. and Malafeyev, O. A.},
	year = {2018},
	pages = {977--999}
}

@article{kolokoltsov_mean-field-game_2016,
	title = {Mean-{Field}-{Game} {Model} for {Botnet} {Defense} in {Cyber}-{Security}},
	volume = {74},
	number = {3},
	journal = {Applied Mathematics \& Optimization},
	author = {Kolokoltsov, V. N. and Bensoussan, A.},
	year = {2016},
	pages = {669--692}
}

@article{cohen_analysis_2023,
	title = {Analysis of the {Finite}-{State} {Ergodic} {Master} {Equation}},
	volume = {87},
	number = {3},
	journal = {Applied Mathematics \& Optimization},
	author = {Cohen, Asaf and Zell, Ethan},
	year = {2023},
	pages = {40},
}

@article{carmona_probabilistic_2021,
	title = {A {Probabilistic} {Approach} to {Extended} {Finite} {State} {Mean} {Field} {Games}},
	volume = {46},
	number = {2},
	journal = {Mathematics of Operations Research},
	author = {Carmona, Ren\'{e} and Wang, Peiqi},
	year = {2021},
	pages = {471--502}
}

@article{cesaroni_stationary_2023,
  title={Stationary {E}quilibria and their {S}tability in a {K}uramoto {MFG} with {S}trong {I}nteraction},
  author={Cesaroni, Annalisa and Cirant, Marco},
  journal={Communications in Partial Differential Equations},
  pages={1--27},
  year={2024}
}

@inproceedings{kuramoto_self-entrainment_1975,
	address = {Berlin, Heidelberg},
	series = {Lecture {Notes} in {Physics}},
	title = {Self-entrainment of a population of coupled non-linear oscillators},
	booktitle = {International {Symposium} on {Mathematical} {Problems} in {Theoretical} {Physics}},
	publisher = {Springer},
	author = {Kuramoto, Yoshiki},
	editor = {Araki, Huzihiro},
	year = {1975},
	pages = {420--422}
}

@article{carmona_jet_2020,
	title = {Jet {Lag} {Recovery}: {Synchronization} of {Circadian} {Oscillators} as a {Mean} {Field} {Game}},
	volume = {10},
	number = {1},
	journal = {Dynamic Games and Applications},
	author = {Carmona, Ren\'{e} and Graves, Christy V.},
	year = {2020},
	pages = {79--99}
}

@article{yin_synchronization_2010,
	title = {Synchronization of {C}oupled {O}scillators is a {G}ame},
	author = {Yin, Huibing and Mehta, Prashant G. and Meyn, Sean P. and Shanbhag, Uday V.},
     journal={IEEE Transactions on Automatic Control},
     volume={57},
     number={4},
     pages={920--935},
     year={2011}
}

@article{carmona_synchronization_2023,
	title = {Synchronization in a {Kuramoto} {Mean} {Field} {Game}},
	volume = {48},
	number = {9},
	journal = {Communications in Partial Differential Equations},
	author = {Carmona, Ren\'{e} and Cormier, Quentin and Soner, H Mete},
	year = {2023},
	pages = {1214--1244}
}

@article{gomes_continuous_2013,
	title = {Continuous {Time} {Finite} {State} {Mean} {Field} {Games}},
	volume = {68},
	number = {1},
	journal = {Applied Mathematics \& Optimization},
	author = {Gomes, Diogo A. and Mohr, Joana and Souza, Rafael Rig\~{a}o},
	year = {2013},
	pages = {99--143}
}

@inproceedings{YMMS,
    title={Bifurcation {A}nalysis of a {H}eterogeneous {M}ean-{F}ield {O}scillator{G}ame {M}odel},
    author={Yin, Huibing and Mehta, Prashant G and Meyn, Sean P and Shanbhag, Uday V},
    booktitle={2011 50th IEEE Conference on Decision and Control and European Control Conference},
    pages={3895--3900},
    year={2011}
  }

@article{lions,
  title={Cours au {C}oll\`{e}ge de {F}rance},
  author={Lions, Pierre-Louis},
  journal={Available at www.college-de-france.fr},
  year={2007}
}

@book{fleming_soner,
  title={Controlled {M}arkov {P}rocesses and {V}iscosity {S}olutions},
  author={Fleming, Wendell and Soner, H.\ Mete},
  volume={25},
  year={2006},
  publisher={Springer Science \& Business Media}
}

@article{cohen2024asymptotic,
  title={Asymptotic {N}ash {E}quilibria of {F}inite-{S}tate {E}rgodic {M}arkovian {M}ean {F}ield {G}ames},
  author={Cohen, Asaf and Zell, Ethan},
  journal={arXiv preprint arXiv:2404.11695},
  year={2024}
}

@article{briani2018stable,
  title={Stable {S}olutions in {P}otential {M}ean {F}ield {G}ame {S}ystems},
  author={Briani, Ariela and Cardaliaguet, Pierre},
  journal={Nonlinear Differential Equations and Applications},
  volume={25},
  pages={1--26},
  year={2018},
  publisher={Springer}
}

\end{document}